\documentclass[reqno]{amsart}
\usepackage{fullpage,graphicx}
\usepackage{amsmath,amsthm}
\usepackage{amsfonts,amssymb}
\usepackage{newlfont,mathrsfs,mathtools}
\usepackage{txfonts}
\usepackage{etoolbox}

\makeatletter
\let\old@tocline\@tocline
\let\section@tocline\@tocline
\newcommand{\subsection@dotsep}{4.5}
\newcommand{\subsubsection@dotsep}{4.5}
\patchcmd{\@tocline}
{\hfil}
{\nobreak
	\leaders\hbox{$\m@th
		\mkern \subsection@dotsep mu\hbox{.}\mkern \subsection@dotsep mu$}\hfill
	\nobreak}{}{}
\let\subsection@tocline\@tocline
\let\@tocline\old@tocline

\patchcmd{\@tocline}
{\hfil}
{\nobreak
	\leaders\hbox{$\m@th
		\mkern \subsubsection@dotsep mu\hbox{.}\mkern \subsubsection@dotsep mu$}\hfill
	\nobreak}{}{}
\let\subsubsection@tocline\@tocline
\let\@tocline\old@tocline

\let\old@l@subsection\l@subsection
\let\old@l@subsubsection\l@subsubsection

\def\@tocwriteb#1#2#3{%
	\begingroup
	\@xp\def\csname #2@tocline\endcsname##1##2##3##4##5##6{%
		\ifnum##1>\c@tocdepth
		\else \sbox\z@{##5\let\indentlabel\@tochangmeasure##6}\fi}%
	\csname l@#2\endcsname{#1{\csname#2name\endcsname}{\@secnumber}{}}%
	\endgroup
	\addcontentsline{toc}{#2}%
	{\protect#1{\csname#2name\endcsname}{\@secnumber}{#3}}}%

\newlength{\@tocsectionindent}
\newlength{\@tocsubsectionindent}
\newlength{\@tocsubsubsectionindent}
\newlength{\@tocsectionnumwidth}
\newlength{\@tocsubsectionnumwidth}
\newlength{\@tocsubsubsectionnumwidth}
\newcommand{\settocsectionnumwidth}[1]{\setlength{\@tocsectionnumwidth}{#1}}
\newcommand{\settocsubsectionnumwidth}[1]{\setlength{\@tocsubsectionnumwidth}{#1}}
\newcommand{\settocsubsubsectionnumwidth}[1]{\setlength{\@tocsubsubsectionnumwidth}{#1}}
\newcommand{\settocsectionindent}[1]{\setlength{\@tocsectionindent}{#1}}
\newcommand{\settocsubsectionindent}[1]{\setlength{\@tocsubsectionindent}{#1}}
\newcommand{\settocsubsubsectionindent}[1]{\setlength{\@tocsubsubsectionindent}{#1}}

\renewcommand{\l@section}{\section@tocline{1}{\@tocsectionvskip}{\@tocsectionindent}{}{\@tocsectionformat}}%
\renewcommand{\l@subsection}{\subsection@tocline{2}{\@tocsubsectionvskip}{\@tocsubsectionindent}{}{\@tocsubsectionformat}}%
\renewcommand{\l@subsubsection}{\subsubsection@tocline{3}{\@tocsubsubsectionvskip}{\@tocsubsubsectionindent}{}{\@tocsubsubsectionformat}}%
\newcommand{\@tocsectionformat}{}
\newcommand{\@tocsubsectionformat}{}
\newcommand{\@tocsubsubsectionformat}{}
\expandafter\def\csname toc@1format\endcsname{\@tocsectionformat}
\expandafter\def\csname toc@2format\endcsname{\@tocsubsectionformat}
\expandafter\def\csname toc@3format\endcsname{\@tocsubsubsectionformat}
\newcommand{\settocsectionformat}[1]{\renewcommand{\@tocsectionformat}{#1}}
\newcommand{\settocsubsectionformat}[1]{\renewcommand{\@tocsubsectionformat}{#1}}
\newcommand{\settocsubsubsectionformat}[1]{\renewcommand{\@tocsubsubsectionformat}{#1}}
\newlength{\@tocsectionvskip}
\newcommand{\settocsectionvskip}[1]{\setlength{\@tocsectionvskip}{#1}}
\newlength{\@tocsubsectionvskip}
\newcommand{\settocsubsectionvskip}[1]{\setlength{\@tocsubsectionvskip}{#1}}
\newlength{\@tocsubsubsectionvskip}
\newcommand{\settocsubsubsectionvskip}[1]{\setlength{\@tocsubsubsectionvskip}{#1}}

\patchcmd{\tocsection}{\indentlabel}{\makebox[\@tocsectionnumwidth][l]}{}{}
\patchcmd{\tocsubsection}{\indentlabel}{\makebox[\@tocsubsectionnumwidth][l]}{}{}
\patchcmd{\tocsubsubsection}{\indentlabel}{\makebox[\@tocsubsubsectionnumwidth][l]}{}{}

\newcommand{\@sectypepnumformat}{}
\renewcommand{\contentsline}[1]{%
	\expandafter\let\expandafter\@sectypepnumformat\csname @toc#1pnumformat\endcsname%
	\csname l@#1\endcsname}
\newcommand{\@tocsectionpnumformat}{}
\newcommand{\@tocsubsectionpnumformat}{}
\newcommand{\@tocsubsubsectionpnumformat}{}
\newcommand{\setsectionpnumformat}[1]{\renewcommand{\@tocsectionpnumformat}{#1}}
\newcommand{\setsubsectionpnumformat}[1]{\renewcommand{\@tocsubsectionpnumformat}{#1}}
\newcommand{\setsubsubsectionpnumformat}[1]{\renewcommand{\@tocsubsubsectionpnumformat}{#1}}
\renewcommand{\@tocpagenum}[1]{%
	\hfill {\mdseries\@sectypepnumformat #1}}

\let\oldappendix\appendix
\renewcommand{\appendix}{%
	\leavevmode\oldappendix%
	\addtocontents{toc}{%
		\protect\settowidth{\protect\@tocsectionnumwidth}{\protect\@tocsectionformat\sectionname\space}%
		\protect\addtolength{\protect\@tocsectionnumwidth}{2em}}%
}
\makeatother

\makeatletter
\settocsectionnumwidth{2em}
\settocsubsectionnumwidth{2.5em}
\settocsubsubsectionnumwidth{3em}
\settocsectionindent{1pc}%
\settocsubsectionindent{\dimexpr\@tocsectionindent+\@tocsectionnumwidth}%
\settocsubsubsectionindent{\dimexpr\@tocsubsectionindent+\@tocsubsectionnumwidth}%
\makeatother

\settocsectionvskip{0pt}
\settocsubsectionvskip{0pt}
\settocsubsubsectionvskip{0pt}

\settocsectionformat{\bfseries}
\settocsubsectionformat{\mdseries}
\settocsubsubsectionformat{\mdseries}
\setsectionpnumformat{\bfseries}
\setsubsectionpnumformat{\mdseries}
\setsubsubsectionpnumformat{\mdseries}

\let\oldtableofcontents\tableofcontents
\renewcommand{\tableofcontents}{%
	\vspace*{-\linespacing}
	\oldtableofcontents}

\allowdisplaybreaks

\usepackage[colorlinks,
linkcolor=blue,
anchorcolor=blue,
citecolor=blue
]{hyperref}

\numberwithin{equation}{section}
\newtheorem{thm}{Theorem}[section]
\newtheorem{pp}[thm]{Proposition}
\newtheorem{lm}[thm]{Lemma}
\newtheorem{df}[thm]{Definition}

\newtheorem{Rmk}[thm]{Remark}
\newcommand{\R}{\mathbb{R}}
\newcommand{\rd}{\mathrm{d}}
\newcommand{\cQ}{\mathcal{Q}}
\newcommand{\cR}{\mathcal{R}}
\newcommand{\cI}{\mathcal{I}}
\newcommand{\cJ}{\mathcal{J}}

\newcommand{\mS}{\mathbb{S}}
\newcommand{\T}{\mathbb{T}}
\newcommand{\N}{\mathbb{N}}
\newcommand{\Z}{\mathbb{Z}}

\newcommand{\Id}{\operatorname{Id}}
\newcommand{\Div}{\operatorname{div}}
\newcommand{\tr}{\operatorname{tr}}
\newcommand{\supp}{\operatorname{supp}}

\newcommand{\tu}{\tilde{u}}

\newcommand{\PL}{P_{\leqslant\ell^{-1}}}

\newcommand{\UL}{U_{\leqslant\ell^{-1}}}
\newcommand{\UG}{U_{>\ell^{-1}}}
\newcommand{\PLq}{P_{\leqslant\ell_{q,i}^{-1}}}
\newcommand{\PGq}{P_{>\ell_{q,i}^{-1}}}
\newcommand{\ULq}{U_{\leqslant\ell_{q,i}^{-1}}}
\newcommand{\UGq}{U_{>\ell_{q,i}^{-1}}}
\newcommand{\PLo}{P_{\leqslant\ell_{q,0}^{-1}}}

\newcommand{\ULo}{U_{\leqslant\ell_{q,0}^{-1}}}

\newcommand{\nb}{\nabla}
\newcommand{\pa}{\partial}


\begin{document}
	\title[] {Non-uniqueness for the nonlinear dynamical Lam\'e system}
	
	\author{Shunkai Mao}
	\address{School of Mathematical Sciences, Fudan University, China.}
	\email[Shunkai Mao]{21110180056@m.fudan.edu.cn}
	
	\author{Peng Qu}
	\address{School of Mathematical Sciences $\&$ Shanghai Key Laboratory for Contemporary Applied Mathematics, Fudan University, China.}
	\email[Peng Qu]{pqu@fudan.edu.cn}
	\thanks{}
	
	\begin{abstract}
	We consider the Cauchy problem for the nonlinear dynamical Lam\'e system with double wave speeds in a $d$-dimensional $(d=2,3)$ periodic domain. Moreover, the equations can be transformed into a linearly degenerate hyperbolic system.  We could construct infinitely many  continuous  solutions in $C^{1,\alpha}$ emanating from the same small initial data for $\alpha<\frac{1}{60}$.  The proof relies on the convex integration scheme. We construct a new class of building blocks with compression structure by using the double wave speeds characteristic of the equations.  
	\end{abstract}
	\maketitle
	
	{
		\tableofcontents
	}
\section{Introduction}
In this paper, we consider the Cauchy problem for the nonlinear dynamical Lam\'e system on $[0,T]\times\T^d,$ $T>0,$ $\T=[-\pi, \pi],$ $d=2,3,$ 
\begin{equation}\label{system}
	\left\{
	\begin{aligned}
		&\partial_{tt}u-\mu\Delta u-(\lambda+\mu)\nabla\Div u+\Div(\tr(\nabla u(\nabla u)^{\top})\Id-(\nabla u)^\top\nabla u)=0,&& (t,x)\in (0,T)\times\T^d,  \\
		&u(0,x)=u_0(x),\quad \partial_tu(0,x)=u_1(x),&& x\in \T^d,
	\end{aligned}
	\right.
\end{equation}
where $u:[0,T]\times\T^d\rightarrow \R^d$ is a vector function, $\lambda$ and $\mu$ are constants satisfying $\mu>0, \lambda+\mu>0$. \par 
Similar to \cite{Boc09,CMT24}, we introduce next the definition of a weak solution of this system.
\begin{df}[Weak Solution]\label{def of weak solution}
	By a weak solution of \eqref{system}, defined on some interval $(0,T)$, we mean a function $u\in C([0,T]; H^2(\T^d))\bigcap C^1([0,T]; H^1(\T^d))$, with $\pa_tu\in C([0,T]; H^1(\T^d))$, $\pa_{tt}u\in L^2(0,T; L^2(\T^d))$,   $u(0,x)=u_0(x)\in H^2(\T^d)$, and $\pa_tu(0,x)=u_1(x)\in H^1(\T^d)$ such that
	\begin{equation}\label{eq of weak solution}
	\begin{aligned} \int_0^T\int_{\T^d}(-\partial_tu\cdot\partial_{t}\eta+\mu\nabla u:\nabla\eta+(\lambda+\mu)\Div u\Div\eta-(\tr(\nabla u(\nabla u)^{\top})\Id-(\nabla u)^\top\nabla u):\nabla\eta)\rd x\rd t&\\
	=\int_{\T^d}u_1(x)\cdot\eta(0,x)\rd x&,
	\end{aligned}
	\end{equation}
for any smooth function $\eta\in C_0^{\infty}([0,T)\times\T^d;\R^d)$.
\end{df}
Moreover, we consider weak solutions which have H\"{o}lder derivatives in space, for instance,
\begin{equation}
	|\partial_t u(t,x)-\partial_tu(t,y)|+|\nabla u(t,x)-\nabla u(t,y)|\leqslant C|x-y|^\beta,\qquad \forall x,y\in\T^3,\forall t\in[0,T],
\end{equation}
for some constant $C$ which is independent of $t$. Here $\beta\in(0,1)$ is the H\"{o}lder index.  In this paper, we aim to demonstrate the non-uniqueness of weak solutions that belong to $C^{1,\alpha}([0,T] \times \T^d)$ for some $\alpha>0$, for the nonlinear dynamical Lam\'e system as described in \eqref{system}. \par
\begin{Rmk}
Similar to the study \cite{BDSV19} on the non-uniqueness of weak solutions in $C^{\frac{1}{3}-}([0,T] \times \T^d)$ of the Euler equations, here we also consider weak solutions in H\"{o}lder  spaces. 
\end{Rmk} 
As for the general wave equations in $n$ dimensions, 
\begin{equation}\label{nonlinear wave equation}
	\left\{
	\begin{aligned}
		&\pa_{tt}u-\Delta u=F(u,\nb u,\nb^2u), \quad t>0,x\in \R^n, \\
		&u(0,x)=\varepsilon\psi(x)\in C_c^\infty(\R^n),\quad \pa_tu(0,x)=\varepsilon\eta(x)\in C_c^\infty(\R^n),
	\end{aligned}
	\right.
\end{equation}
where $\varepsilon$ is a small parameter, $\tilde{\Lambda}=(u,\nb u,\nb^2u)$, and $F(\tilde{\Lambda})=O(|\tilde{\Lambda}|^{1+\alpha})$ for $\alpha\in \Z_+$. In \cite{LZ17}, Li--Zhou gave the lower bound estimate of the life-span $\tilde{T}(\varepsilon)$ of the classical solution to $\eqref{nonlinear wave equation}$, \par
\begin{equation}\label{life span}
	\tilde{T}(\varepsilon)\geqslant\left\{\begin{aligned} &be(\varepsilon)(\text{in}\ \text{\cite{LZ94}}),&&n=2,\alpha=1,\\
		&b\varepsilon^{-2}(\text{in}\ \text{\cite{LY89,LY91,Lin90}}), &&n=3,\alpha=1,
	\end{aligned}
	\right.	
\end{equation}
where $e(\varepsilon)$ is defined by $\varepsilon^2e^2(\varepsilon)\ln(1+e(\varepsilon))=1$. Moreover, if the nonlinear term $F$ doesn't depend on $u$ explicitly: $F=F(\nb u,\nb^2u)$, the result can be improved into
\begin{equation}\label{life span imporved}
	\tilde{T}(\varepsilon)\geqslant\left\{\begin{aligned} &b\varepsilon^{-2}(\text{in}\ \text{\cite{LZ94}}),&&n=2,\alpha=1,\\
		&e^{a\varepsilon^{-1}}(\text{in}\ \text{\cite{LZ92,Lin90}}), &&n=3,\alpha=1.
	\end{aligned}
	\right.	
\end{equation}

The nonlinear term in \eqref{system} can be shown to satisfy the null condition, a concept first introduced by Klainerman. In the three-dimensional case, Christodoulou \cite{Chr86} and Klainerman \cite{Kla86} proved that when $ F $ satisfies the null condition, the Cauchy problem \eqref{nonlinear wave equation} admits a unique global classical solution for sufficiently small $\varepsilon$. For the two-dimensional case, Alinhac \cite{Ali01} demonstrated almost global existence for quasi-linear wave equations under the null condition. Later, Zha \cite{Zha19} extended this work to a broader class of systems using a unified methodology. 

On the other hand, for nonlinear wave equations with small initial data that do not satisfy the null condition, extensive studies have focused on finite-time blowup phenomena. Notable contributions in this area include works such as \cite{Ali99,Ali99B,Ali00,Ali01B,John81,Sid84,Sid85}.

It can be observed that \eqref{system} describes a system of quasi-linear wave equations with double wave speeds, a characteristic that closely resembles elastodynamics. Significant progress has been made in the study of elastodynamics, as highlighted in works such as \cite{John88,Lei15,Lei16,Sid96,Sid00,ST05,Wang17}. John \cite{John88} demonstrated the almost global existence of elastic waves with finite amplitude originating from small initial perturbations. Later, Sideris \cite{Sid96,Sid00} proved that the null condition guarantees the global existence of nonlinear elastic waves in three spatial dimensions. Additionally, Lei \cite{Lei16} introduced the concept of the strong null condition and established that systems of incompressible isotropic Hookean elastodynamics in two dimensions admit a unique global classical solution for sufficiently small initial displacements.

Significant progress has been made in the study of low-regularity solutions to systems of quasilinear wave equations, as reflected in works such as \cite{ACY23,Lei08,Lin98,ST05A,Wang17,ZhaH20,Zhang24}. It has been established in \cite{Lin98,ST05A,Wang17} that the Cauchy problem for quasilinear wave equations is locally well-posed in $ H^s(\mathbb{R}^3) \times H^{s-1}(\mathbb{R}^3) $ for $ s > 3 $, which is generally considered sharp.  Zha--Hidano studied the Cauchy problem for 3D quasilinear wave systems satisfying the null condition with low-regularity initial data. In the radially symmetric case, they demonstrated the global existence for small data in $ H^3(\mathbb{R}^3) \times H^2(\mathbb{R}^3) $ with a low weight and applied their results to 3D nonlinear elastic waves. An--Chen--Yin \cite{ACY23} extended Lindblad's classical results \cite{Lin98} on the scalar wave equation by showing that the Cauchy problem for 3D elastic waves, a physical system with multiple wave speeds, is ill-posed in $H^3(\mathbb{R}^3) $ due to instantaneous shock formation. 

The Cauchy problem \eqref{nonlinear wave equation} with small initial data can essentially be reduced to the Cauchy problem for a system of quasi-linear hyperbolic equations. According to \cite{Daf16,Lax57}, for systems with genuinely nonlinear characteristic families, different solutions can be constructed in the absence of an entropy admissibility condition.  In this work, we present a constructive proof of non-uniqueness for a specific class of linearly degenerate systems.

In particular, there have been several significant studies on the  Lam\'e system, such as: Belishev--Lasiecka \cite{BL02} deal with the issue of boundary approximate controllability and related unique continuation
property for a system of dynamic elasticity governed by the Lam\'e model. Ma--Mesquita--Seminario--Huertas \cite{MMSH21} prove an existence result for Lamé systems with a damping-delay component and a nonlinear forcing. Wang--Freitas--Feng--Ramos \cite{WFFR17} investigated the global attractors and synchronization phenomenon of a coupled critical Lamé system defined on a smooth bounded domain $\Omega \subset \R^3$ with nonlinear damping and nonlinear forces of critical growth. In \cite{GH24}, Guesmia--Harkat focused on the asymptotic behaviour of the solution as $t$ and the state variable domain become very large. Different rates of convergence are established
according to the growth of the domain.

Our proof builds upon the convex integration method, originally developed by De Lellis and Székelyhidi \cite{DS09, DS13}. In their foundational works \cite{DS09, DS10}, they established non-uniqueness for incompressible Euler equations in $ L^\infty(\mathbb{R}_x^n \times \mathbb{R}_t; \mathbb{R}^n) $, showing that for certain bounded, compactly supported initial data, no canonical energy admissibility criterion uniquely selects a weak solution.  Building on these ideas, a series of remarkable developments \cite{Buc15, BDIS15, BDS16, DS17, DS13, DS14} culminated in the resolution of Onsager's conjecture. This conjecture posited that the exponent \(\alpha = 1/3\) marks the threshold for energy conservation in weak solutions within the Hölder space \(C^\alpha\) for the incompressible Euler equations. Isett \cite{Isett18} resolved the conjecture, and Buckmaster--De Lellis--Székelyhidi--Vicol \cite{BDSV19} extended the result to the dissipative case.  

Recent advancements have further deepened our understanding of the non-uniqueness of entropy solutions for both compressible and incompressible Euler equations in Hölder spaces. Notably, De Lellis--Kwon \cite{DK22} discovered continuous entropy solutions of the incompressible Euler equations in the Hölder class \(C^{\frac{1}{7}-}\) that satisfy the entropy inequality and strictly dissipate total kinetic energy. The case of the compressible Euler equations was subsequently explored by Giri--Kwon in \cite{GK22}. This approach has also been applied to other fluid dynamics systems, including the incompressible Navier-Stokes equations, the compressible Euler-Maxwell equations, and magnetohydrodynamics (MHD) equations, as demonstrated in works such as \cite{BBV20, BV19, CL22, LQZZ22, LZZ22, LQ20, LX20, MQ23, MY22, NV23}.

\par
There have also been a series of studies \cite{AKKMM20,BKM21,CDK15Re,CK18,CKMS21,KKMM20} on the non-uniqueness of weak solutions to the Riemann problem for various types of Euler equations. In \cite{CDK15Re}, Chiodaroli--De Lellis--Kreml examined the isentropic compressible Euler system in two spatial dimensions with the pressure law \( p(\rho) = \rho^2 \) and demonstrated the non-uniqueness of weak solutions for classical Riemann data. Building on this framework and leveraging the method from \cite{DS10}, Chiodaroli--Kreml \cite{CK18} proved the non-uniqueness of admissible weak solutions to the Riemann problem for isentropic Euler equations. Notably, in \cite{BKM21}, Kreml investigated the non-uniqueness in the multi-dimensional model of Chaplygin gas, showing that all three characteristic families of the system are linearly degenerate. 

This paper is the first to apply the convex integration method to prove non-uniqueness for the nonlinear dynamical Lamé system. The system studied here exhibits a different mathematical structure compared to the Euler equations, particularly due to the characteristic feature of double wave speeds, which is utilized to construct the corresponding building blocks.

In this paper, we construct non-unique solutions in \( C^{1,\frac{1}{60}-}([0,T] \times \mathbb{T}^3) \) and \( C^{1,\frac{1}{30}-}([0,T] \times \mathbb{T}^2) \) for the Cauchy problem \eqref{system}. The proof for the three-dimensional case is slightly more intricate than for the two-dimensional case, so we focus primarily on the three-dimensional case here. In fact, since the solution for the two-dimensional case can be regarded as a special case of the three-dimensional solution, the three-dimensional result serves as a corollary of the two-dimensional one.

\subsection{Main results}
Next, we will introduce the main result of this paper. We present two main theorems that imply the non-uniqueness of weak solutions of \eqref{system} in the H{\"o}lder class.
\begin{thm}\label{thm 1}
	If $\mu>0,$ and $\lambda+\mu>0$, for any $0< \alpha< \frac{1}{60},$ we can find
	infinitely many distinct weak solutions $u\in C^{1, \alpha}([0, T]\times \T^3)$ to the Cauchy problem \eqref{system} emanating from the same small initial data.
\end{thm}
\begin{thm}
	\label{thm 2}
	If $\mu>0,$ and $\lambda+\mu>0$, for any $0< \alpha< \frac{1}{30},$ we can find
	infinitely many distinct weak solutions $u\in C^{1, \alpha}([0, T]\times \T^2)$ to the Cauchy problem \eqref{system} emanating from the same small initial data.
\end{thm}
In this paper, we adapt the convex integration scheme to the quasi-linear wave equations \eqref{system}.  Below, we will introduce the proof process and highlight the differences from the compressible and incompressible Euler equations.\par 
For quasi-linear wave equations, we consider the weak solutions of \eqref{system} emanating from small initial data, with the magnitude of the initial data  $\|u\|_1 $ controlled by a small parameter $\varepsilon$. During the proof process, we would find that $\varepsilon$ depends on the difference of double wave speeds $\lambda+\mu$. In the convex integration scheme, we will construct a series of approximate solutions which converge to the weak solution of \eqref{system}. A crucial step in this process is the construction of the perturbation. This step is also where the differences between the quasi-linear wave equations and the Euler equations arise.\par  
For the Euler equations, we typically aim to construct perturbations that approximately satisfy the transport equation to obtain good estimates on the transport error. However, after some simple calculations, we find that the perturbations used in this paper need to approximately satisfy specific wave equations. Moreover, based on the truncation technique, at $q+1$ step, we can treat the $u_{q}$ from the previous step and its derivatives as constants in the truncated region, so that the equation to be satisfied by the perturbation can be reduced to a wave equation with constant coefficients (see Section \ref{Definition of the perturbation}).  Based on this observation, we propose Lemma \ref{construction of building blocks} to construct new building blocks which consists of a longitudinal wave of size $O(1)$ and a transverse wave of size $O(\varepsilon)$. Notably, we can construct such building blocks only if the wave speed difference  $\lambda+\mu>0,$ highlighting the importance of the double wave speed property. \par 
The construction of the perturbation is also closely related to the form of the nonlinear terms. The nonlinear terms in \eqref{system} differ from those in the Euler equations. In \eqref{system}, the nonlinear term $\Div(\tr(\nabla u(\nabla u)^{\top})\Id-(\nabla u)^\top\nabla u)$  guarantees that the equation satisfies the null condition and that the corresponding hyperbolic system is linearly degenerate. Moreover, it guarantees a certain geometric structure so that by \eqref{Geometric lemma 1}, we can use the low frequency part of $\tr(\nabla \tu_{q,i+1}(\nabla \tu_{q,i+1})^{\top})\Id-(\nabla \tu_{q,i+1})^\top\nabla \tu_{q,i+1}$ to eliminate the Reynolds error $R_q$.  It is also due to this form that waves in different directions produce low--high frequency terms that cannot be eliminated and are poor in estimation. So we can only add perturbation in one direction each time. Then, for the 3D case, at each step of the iteration, we need to add different high frequency waves to the approximate solution for six times (see  Section \ref{Outline of the induction scheme}). This method was used by Luo and Xin in \cite{LX20}.
\subsection{Organization of the paper}\label{Outline of the proof}
In Section \ref{Outline of the induction scheme}, we first present the outline of the induction scheme for constructing the approximate solution sequence $(u_q, c_q, R_q)$ and introduce two main propositions, Propositions \ref{Inductive proposition} and \ref{Bifurcating inductive proposition}, which will be used in the proof of Theorem \ref{thm 1}. Next, in Section \ref{Construction of the starting tuple and the perturbation}, we construct the starting tuple $u_0$ and the perturbation $\tilde{u}_{q,i},$ which consists of five parts. The new error $R_{q+1}$ and the corresponding estimates are given in Sections \ref{Definition of the new errors} and \ref{Estimates on the new Reynolds error}. Finally, the proofs of Propositions \ref{Inductive proposition}, \ref{Bifurcating inductive proposition}, and Theorem \ref{thm 1} are detailed in Section \ref{Proof of the main theorem}. The appendix provides proofs or statements of analytical facts used in the proofs of the propositions in the paper.

\section{Outline of the induction scheme}\label{Outline of the induction scheme}
In this paper, we will construct a series of approximate solutions $u_q$ which satisfy the following approximate system and converge to a weak solution $u$ of  \eqref{system}.
\begin{df}
	 A tuple of smooth tensors $(u,c,R)$ is an  approximation solution tensor of equations \eqref{system} as long as it solves the following system in the sense of distribution,
	\begin{equation}
		\partial_{tt}u-\mu\Delta u-(\lambda+\mu)\nabla\Div u+\Div(\tr(\nabla u(\nabla u)^{\top})\Id-(\nabla u)^\top\nabla u)=\Div(R-c\Id), \label{approximation elstro dynamic}
	\end{equation}
	where the Reynolds error $R$ is a $3\times3$ symmetric matrix.
\end{df}
We denote by $\R^{3\times3}$ the space of $3\times 3$ matrices, whereas $\mS^{3\times3}$ denotes corresponding subspace of symmetric matrices.  Moreover, we use the notation $\|R\|=\max_{ij}|R_{ij}|$ and  introduce the following geometric lemma proposed in \cite{DS13}.
\begin{lm}[Geometric Lemma]\label{Geometric Lemma} For every $N\in\N,$ we can choose $\overline{\lambda}>1$ with the following property. There exist pairwise disjoint subsets
	\begin{align*}
		\Lambda_j\subset\{f\in\Z^3||f|=\overline{\lambda}\}, \quad j\in\{1,2, \cdots,N\},
	\end{align*}
	and smooth functions
	\begin{align*}
		\Gamma_f^{(j)}\in C^\infty(B_{r_0}(\Id)), \quad j\in\{1,2, \cdots,N\},f\in \Lambda_j,
	\end{align*}
	such that\\
	(a)$f\in \Lambda_j$ implies $-f\in \Lambda_j$ and $\Gamma_f^{(j)}=\Gamma_{-f}^{(j)},$ then  we could divide $\Lambda_j$ into two parts $\Lambda_j^+$ and $\Lambda_j^-$ which satisfy 
	\begin{align*}
		\Lambda_j=\Lambda_j^+\bigcup\Lambda_j^-, \quad  \Lambda_j^+\bigcap \Lambda_j^-=\emptyset,
	\end{align*}
	(b) for each $K\in B_{r_0}(\Id) ,$ we have the identity
	\begin{align*}
		K=\frac{1}{2}\sum_{f\in \Lambda_j}(\Gamma_f^{(j)}(K))^2\left(\Id-\frac{f}{|f|}\otimes \frac{f}{|f|}\right)=\sum_{f\in \Lambda_j^+}(\Gamma_f^{(j)}(K))^2\left(\Id-\frac{f}{|f|}\otimes \frac{f}{|f|}\right), \quad \forall K\in B_{r_0}(\Id).
	\end{align*}
\end{lm}     
In this paper, we use the following specific formula. With $\overline{\lambda}=\sqrt{2}>1,$ we choose
\begin{align*}
\Lambda&=\{(0, 1, \pm 1),(0, -1, \pm 1),(1,0, \pm1),(-1,0, \pm1),(1, \pm1,0),(-1, \pm1,0)\},\\
\Lambda^+&=\{f_i\}_{i=1}^6=\{(0,1, \pm 1),(1,0, \pm1),(1, \pm1,0)\}.
\end{align*}  
Notice that
\begin{align*}
	\sum_{i=1}^{6}\frac{1}{4}\left(\Id-\frac{f_i}{|f_i|}\otimes \frac{f_i}{|f_i|}\right)=\Id,
\end{align*}
and
\begin{align*}
	&\Id-\frac{f_1}{|f_1|}\otimes \frac{f_1}{|f_1|}=\left(\begin{matrix}
		1&0&0\\
		0&\frac{1}{2}&-\frac{1}{2}\\
		0&-\frac{1}{2}&\frac{1}{2}\\
	\end{matrix}\right),&& \Id-\frac{f_3}{|f_3|}\otimes \frac{f_3}{|f_3|}=\left(\begin{matrix}
	\frac{1}{2}&0&-\frac{1}{2}\\
	0&1&0\\
	-\frac{1}{2}&0&\frac{1}{2}\\
\end{matrix}\right),&&
\Id-\frac{f_5}{|f_5|}\otimes \frac{f_5}{|f_5|}=\left(\begin{matrix}
	\frac{1}{2}&-\frac{1}{2}&0\\
	-\frac{1}{2}&\frac{1}{2}&0\\
	0&0&1\\
\end{matrix}\right), \\
	&\Id-\frac{f_2}{|f_2|}\otimes \frac{f_2}{|f_2|}=\left(\begin{matrix}
	1&0&0\\
	0&\frac{1}{2}&\frac{1}{2}\\
	0&\frac{1}{2}&\frac{1}{2}\\
\end{matrix}\right),&& \Id-\frac{f_4}{|f_4|}\otimes \frac{f_4}{|f_4|}=\left(\begin{matrix}
	\frac{1}{2}&0&\frac{1}{2}\\
	0&1&0\\
	\frac{1}{2}&0&\frac{1}{2}\\
\end{matrix}\right),&&
\Id-\frac{f_6}{|f_6|}\otimes \frac{f_6}{|f_6|}=\left(\begin{matrix}
	\frac{1}{2}&\frac{1}{2}&0\\
	\frac{1}{2}&\frac{1}{2}&0\\
	0&0&1\\
\end{matrix}\right).
\end{align*}
For $r_0=\frac{1}{18}$ and  $K\in B_{\frac{1}{18}}(\Id) ,$ we could represent $K$ as 
	\begin{align}
	K&=\sum_{i=1}^{3}\left(\frac{3K_{ii}+\sum_{j,k=1}^3|\varepsilon_{ijk}|(2K_{jk}-K_{jj})}{4}\right)\left(\Id-\frac{f_{2i}}{|f_{2i}|}\otimes \frac{f_{2i}}{|f_{2i}|}\right)\nonumber\\
	&\quad+\sum_{i=1}^{3}\left(\frac{3K_{ii}-\sum_{j,k=1}^3|\varepsilon_{ijk}|(2K_{jk}+K_{jj})}{4}\right)\left(\Id-\frac{f_{2i-1}}{|f_{2i-1}|}\otimes \frac{f_{2i-1}}{|f_{2i-1}|}\right)\nonumber\\
	&=\sum_{i=1}^6\Gamma_{f_i}^2(K)\left(\Id-\frac{f_i}{|f_i|}\otimes \frac{f_i}{|f_i|}\right).\label{Geometric lemma 1}
\end{align}
Moreover, for each symmetric tensor $R\in \mathbb{S}^{3\times 3},$ we have
\begin{align*}
	\frac{c}{r_0}\Id-R=\sum_{i=1 }^6\left(\left(\frac{c}{r_0}\right)^{\frac{1}{2}}\Gamma_{f_i}\left(\Id-\frac{r_0}{c}R\right)\right)^2\left(\Id-\frac{f_i}{|f_i|}\otimes \frac{f_i}{|f_i|}\right), \quad \text{with}\ c>\|R\|.
\end{align*}
For convenience, we introduce the following notation:
\begin{itemize}
	\item $\cI + \sigma$ is the concentric enlarged interval $(a-\sigma, b+\sigma)$ when $\cI = [a,b]$.
	\item Furthermore, for the sake of convenience, in what follows, we use the notation $A \lesssim_{\kappa} B$ to mean $A \leqslant CB,$ where $C > 0$ may depend on some fixed constants or functions $\kappa$. Especially, we use the notation $A\lesssim_{N,r} B$ without pointing out the dependence of the implicit constant $C,$ and $N,r\in\N$ can be chosen to have $N\leqslant N^*, r\leqslant r^*$ for some positive integers $N^*,r^*$.    We will not repeatedly specify this.
\end{itemize}
Next, we give some parameters to measure the size of our approximate solutions,
\begin{equation}
	\begin{aligned}
		&\lambda_q=2^{6\lceil b^q\log_2a\rceil}, &&\lambda_{q,i}=\lambda_{q}^{1-\frac{i}{6}}\lambda_{q+1}^{\frac{i}{6}},
		\quad 0\leqslant i\leqslant 6, \\
		&\delta_q=\lambda_q^{-2\beta}, &&\ c_q=\sum_{j=q+1}^{\infty}\delta_j,
	\end{aligned}\label{def of parameter}
\end{equation}
where $a>0$ is a large parameter and $b>1,\ \beta>0$. Due to truncation and smoothing, the domains of definition for the approximate solutions $u_{q,i}$ change at each step, here we choose it at step $q$ as $\cI^{q,i-1}=[-\tau_{q,i-1},T+\tau_{q,i-1}],$ where $\tau_{q,-1}=\tau_{q-1,5}$ and
$$
\tau_{q,i}=\left(\lambda_{q,i}\lambda_{q,i+1}\right)^{-\frac{1}{2}}\delta_{q}^{-\frac{1}{4}}, \quad q\geqslant 0,\ 0\leqslant i\leqslant 5.
$$\par
At each step, we give a correction to make the error $R_q$ get smaller which  converge to zero (in\ H\"{o}lder space) as $q$ goes to infinity. 
We assume the following inductive estimates on $(u_q,R_q)$ satisfying \eqref{approximation elstro dynamic}.
\begin{align}
	\|u_q\|_{0}, \|u_q\|_{1}, \|\partial_tu_q\|_{0}\leqslant \varepsilon-\delta_q^{\frac{1}{2}}, \quad\|\partial_t^ru_q\|_{N}\leqslant M\lambda_{q}^{N+r-1}\delta_{q}^{\frac{1}{2}}, \quad  2\leqslant N+r\leqslant 3, \label{est on u_q}
\end{align}
and
\begin{align}
	\|\partial_t^rR_q\|_{N}\leqslant \lambda_{q}^{N+r-2\gamma }\delta_{q+1}, \quad 0\leqslant N+r\leqslant 2, \label{est on R_q}
\end{align}
where $\|\cdot\|_N=\|\cdot\|_{C^0(\cI^{q,-1};C^N(\T^3))},$ $\gamma=\gamma(\beta)=\frac{1}{3}\left(\frac{1}{24} (1+12 \beta)-\sqrt{\frac{\beta-6 \beta^2}{6}}\right)<\frac{1}{36},$ $\varepsilon\ll 1,$ and $M=M(\lambda, \mu)>1$.\par 
Similar to \cite{DK22,GK22}, we give two inductive propositions.
\begin{pp}[Inductive proposition]\label{Inductive proposition}
	For any constants $\beta \in(0, \frac{1}{60}),\ \mu>0,$ and $\lambda+\mu>0,$ there exist constants  $\varepsilon_1=\varepsilon_1(\lambda, \mu)>0,$  $0<\gamma(\beta)<1/36,$ $\bar{b}(\beta) > 1,$ $M=M(\lambda, \mu)> 1,$ and $a^*_0=a^*_0(\beta, \bar{b}, \gamma, M) > 0$ such that the following property holds. Let $c_q=\sum_{j=q+1}^\infty\delta_j$ and $b=\bar{b},$ for any $a >a^*_0$ and $\varepsilon<\varepsilon_1,$ assume that $\left(u_q,c_q,R_q\right)$  is a solution of \eqref{approximation elstro dynamic} defined on the time interval $[-\tau_{q,-1}, T+{\tau_{q,-1}}]$ satisfying \eqref{est on u_q} and \eqref{est on R_q}. Then, we can find a corrected approximation solution $\left(u_{q+1}, c_{q+1}, R_{q+1}\right)$ which is defined on the time interval $[-\tau_{q+1,-1}, T+{\tau_{q+1,-1}}],$ satisfies \eqref{est on u_q}--\eqref{est on R_q} for $q+1$ and additionally
	\begin{equation}\label{Proposition of induction}
		\begin{aligned}
			\sum_{0\leqslant N+r\leqslant3}\lambda_{q+1}^{1-N-r}\|\partial_{t}^r(u_{q+1}-u_q)\|_{C^0\left([0, T] ; C^N\left(\mathbb{T}^3\right)\right)}\leqslant M \delta_{q+1}^{\frac{1}{2}}.
		\end{aligned}
	\end{equation}
\end{pp}
\begin{pp}[Bifurcating inductive proposition]\label{Bifurcating inductive proposition}
	Let the constants $\lambda,$ $\mu,$ parameters  $\varepsilon_1,$ $\beta,$ $M,$ $\bar{b},\gamma,$ and $a^*_0,$ and the tuple $\left(u_q, c_q, R_q\right)$ be given as in the statement of Proposition \ref{Inductive proposition}. For any time interval $\cI \subset(0, T)$ which satisfies $|\cI| \geqslant 3 \tau_{q,-1},$ we can produce two different tuples $\left(u_{q+1}, c_{q+1}, R_{q+1} \right)$ and $\left(\overline{u}_{q+1}, c_{q+1}, \overline{R}_{q+1}\right)$ which share the same initial data, satisfy the same conclusions of Proposition \ref{Inductive proposition}  and additionally
	\begin{equation}\label{Bifurcating proposition}
		\begin{aligned}
			\|\overline{u}_{q+1}-u_{q+1}\|_{C^0\left([0, T] ; L^2\left(\mathbb{T}^3\right)\right)} \geqslant (16\lambda_{q+1})^{-1}\pi^{\frac{3}{2}} \delta_{q+1}^{\frac{1}{2}}, \quad \supp_t\left(\overline{u}_{q+1}-u_{q+1}\right) \subset \cI.
		\end{aligned}
	\end{equation}
	Furthermore, if we are given two tuples $\left(u_q, c_q, R_q\right)$ and $\left(\overline{u}_q, c_q, \overline{R}_q\right)$ satisfying \eqref{est on u_q}--\eqref{est on R_q}, there exists some interval $\cJ \subset(0, T)$ satisfies
	\begin{align}
		\supp_t\left(u_q-\overline{u}_q, R_q-\overline{R}_q\right) \subset \mathcal{J}, \label{supp of diff at q step}
	\end{align}
	and we can exhibit two different tuples $\left(u_{q+1}, c_{q+1}, R_{q+1}\right)$ and $\left(\overline{u}_{q+1}, c_{q+1},  \overline{R}_{q+1}\right)$ satisfying the same conclusions of Proposition \ref{Inductive proposition}. Moreover, the support of their difference satisfies
	\begin{equation}
		\supp_t\left(u_{q+1}-\overline{u}_{q+1}, R_{q+1}-\overline{R}_{q+1}\right) \subset \mathcal{J}+\left(\lambda_q \delta_q^{\frac{1}{4}}\right)^{-1}.\label{supp of diff at q+1 step}
	\end{equation}
\end{pp}
In more detail, at each step of the iteration, we will add different high frequency waves to the approximate solution for six times. Let $u_{q,0}=u_q$ and $R_{q,0}=R_q,$ if we add the correction $\tu_{q,i}$ (with frequency $\lambda_{q,i}$) to $u_{q,i-1}=u_q+\sum_{r=1}^{i-1}\tu_{q,r},$ $i=1,2, \cdots,6,$ the new stress error $R_{q,i}$ takes the form 
\begin{align*}
	R_{q,i}=R_{q,i-1}+(\delta_{q+1}^{\frac{1}{2}}\Gamma_{f_i}(\Id-\delta_{q+1}^{-1}R_{q}))^2\left(\Id-\frac{f_i}{|f_i|}\otimes \frac{f_i}{|f_i|}\right)+\delta R_{q,i}, \quad i\in \{1,2, \cdots,6\},
\end{align*}
where $\delta R_{q,i}$ is a smaller correction. Let $u_{q+1}=u_{q,6}$ and $R_{q+1}=R_{q,6}-\delta_{q+1}\Id$ be the new approximate solution
and the new Reynolds error at  $q+1$ step. Repeating this procedure, the Reynolds errors will converge to zero. Moreover, we assume the following inductive estimates on $(u_{q,i},R_{q,i})$ satisfying \eqref{approximation elstro dynamic}.
\begin{align}
	\|u_{q,i}\|_{0}, \|u_{q,i}\|_{1}, \|\partial_tu_{q,i}\|_{0}\leqslant \varepsilon-(7-i)\delta_{q+1}^{\frac{1}{2}}, \quad\|\partial_t^ru_{q,i}\|_{N}\leqslant M\lambda_{q,i}^{N+r-1}\delta_{q+1}^{\frac{1}{2}}, \quad 2\leqslant N+r\leqslant 3, \label{est on u_qi}
\end{align}
and
\begin{align}
	\|\partial_t^r\delta R_{q,i}\|_{N}\leqslant \lambda_{q,i}^{N+r-2\gamma }\delta_{q+2}, \quad 0\leqslant N+r\leqslant 2, \label{est on R_qi}
\end{align}
where $1\leqslant i\leqslant 6$ and $\|\cdot\|_N=\|\cdot\|_{C^0(\cI^{q,i-1};C^N(\T^3))}$.\par 

The proof for Theorem \ref{thm 1}  relies on the above two important propositions. Proposition \ref{Inductive proposition} provides an iterative framework for constructing a sequence of approximate solutions to \eqref{system}. Specifically, it states that given an initial tuple $(u_0, c_0, R_0),$ we can generate a sequence of tuples $(u_q, c_q, R_q)$ that satisfy \eqref{est on u_q}, \eqref{est on R_q}, and \eqref{Proposition of induction} for all $q \geq 0$. Consequently, we can demonstrate that $u_q$ converges to $u,$ which is a weak solution of \eqref{system} in $C^{1,\alpha}([0,T] \times \T^3),$ with $0<\alpha < \frac{1}{60}$. Furthermore, Proposition \ref{Bifurcating inductive proposition} facilitates the construction of different sequences of tuples that converge to distinct solutions, thereby enabling the construction of infinitely many solutions.\par
\section{Construction of the starting tuple and the perturbation}\label{Construction of the starting tuple and the perturbation} 
\subsection{Construction of the starting tuple}
\label{Construction of the starting tuple}
We will first introduce the construction of the starting tuple. Let
\begin{align*}
	u_0=\frac{\varepsilon\delta_1^{\frac{1}{2}}}{2\lambda_0^{2}|f_1|^2}f_1\left(e^{i\lambda_0(f_1\cdot x-(\lambda+2\mu)^{\frac{1}{2}}|f_1|t)}+e^{-i\lambda_0(f_1\cdot x-(\lambda+2\mu)^{\frac{1}{2}}|f_1|t)}\right),
\end{align*}
which satisfies
\begin{align*}
	\partial_{tt}u_0-\mu\Delta u_0-(\lambda+\mu)\nabla\Div u_0+\Div(\tr(\nabla u_0(\nabla u_0)^{\top})\Id-(\nabla u_0)^\top\nabla u_0)=0.
\end{align*}
So we can choose the Reynolds error $R_0 = 0$. It is straightforward to show that $(u_0, c_0, R_0)$ satisfies \eqref{est on u_q} and \eqref{est on R_q}.
\subsection{Mollification}\label{Mollification}
To solve the loss of temporal and spatial derivatives, in this part, we will show the mollification process for $ (u_{q,i},R_{q,i}) $ at the $q_{th}$ step.\par 
First, we introduce some notation in Fourier analysis described in \cite{GK22} and \cite{Gra08}. We can define the Fourier transform and its inverse of a function $f$ in Schwartz space $\mathcal{S}(\R^3)$ as
$$
\hat{f}(\xi)=\frac{1}{(2\pi)^3}\int_{\R^3}f(x)e^{-ix\cdot\xi}\rd x, \qquad\check{f}(x)=\int_{\R^3}f(\xi)e^{ix\cdot\xi}\rd\xi.
$$
Moreover, the Fourier transform can be extended to linear functionals in $\mathcal{S}^\prime(\R^3)$ which is the dual space of $\mathcal{S}(\R^3)$. Following \cite{DK22,GK22}, we can multiply the Fourier transform of $f$ by a smooth cut-off function, apply the inverse Fourier transform, and obtain a smooth function which is the standard convention for Littlewood-Paley operators. Let $\phi(\xi)$ be a radial smooth function such that $\text{supp}\phi(\xi)\subset B(0,2)$ and $\phi\equiv1$ on $\overline{B(0,1)}$. Then, for any $j\in\Z$ and distribution $f$ on $\R^3,$ we can define
$$
\widehat{P_{\leqslant2^j}f}(\xi):=\phi\left(\frac{\xi}{2^j}\right)\hat{f}(\xi), \qquad\widehat{P_{>2^j}f}(\xi):=\left(1-\phi\left(\frac{\xi}{2^j}\right)\right)\hat{f}(\xi), \qquad
\widehat{P_{2^j}f}(\xi):=\left(\phi\left(\frac{\xi}{2^j}\right)-\phi\left(\frac{\xi}{2^{j-1}}\right)\right)\hat{f}(\xi).
$$
For a given number $a,$ we define $P_{\leqslant a}=P_{\leqslant 2^J}$ and $P_{> a}f=f-P_{\leqslant a}f,$ where $J=\lfloor\log_2a\rfloor$ is the largest integer which satisfies $2^J\leqslant a$. For $\ell>0,$ if $f$ is a spatially periodic function on $[c,d]\times\T^3,$ $\PL f$ can be written as the space convolution of $f$ with kernel $\check{\phi}_\ell(\cdot):=2^{3J}\check{\phi}(2^J\cdot),$ where $J=\lfloor-\log_2\ell\rfloor$ is the largest integer which satisfies $2^J\leqslant \ell^{-1},$ and it is also a spatially periodic function on $[c,d]\times\T^3$. More details can be found in \cite{DK22,GK22,Gra08}. Finally, we present an important inequality, for $k\in \N,$ 
\begin{align}
	\int_{\R^3}|y|^k|\check{\phi}_\ell(y)| \rd  y=2^{-kJ}\int_{\R^3}|y|^k|\check{\phi}(y)| \rd  y\lesssim_k\ell^k. \label{est on  mollification function}
\end{align}
We can also define the mollification in time. Following mollification method in \cite{NV23} and Definition 4.15 in \cite{GKN23}, we could define $\phi^t\in C_c^\infty(\R)$ which satisfies $\text{supp}\phi^t\subset (-1,1)$ and
\begin{equation}
	\int_{\R}\phi^t(\tau)\rd\tau=1, \quad \int_{\R}\phi^t(\tau)\tau^n\rd\tau=0, \quad \forall n=1,2, \cdots,n_0+3,
\end{equation} 
where $ n_0= \left\lceil \frac{12b(1+2\gamma)}{b-1+6\beta} \right\rceil $. 
 Let $\phi^t_\delta(\tau)=\delta^{-1}\phi^t(\delta^{-1}\tau),$ then we  define $\UL f=\phi^t_\ell\ast f$ and $\UG f=f-\UL f$.
\begin{Rmk}
Here, $b$ and $\gamma$ is actually a function of $\beta$. So we could consider $n_0$ as a function of $\beta$.
\end{Rmk}
\begin{Rmk}
If we set $\ell=0$ when using $\PL$ or $\UL,$ we mean $\PL f=f$ and $\UL f= f$.
\end{Rmk}
Next, we introduce the parameter $\ell_{q,i},$ defined by 
\begin{equation}\label{def of l}
	\ell_{q,i}=\left(\lambda_{q,i}^{\frac{1}{2}} \lambda_{q,i+1}^{\frac{1}{2}}\right)^{-1}\delta_{q}^{-\frac{1}{4}}, \quad 0\leqslant i\leqslant 5,
\end{equation} 
and give the regularized terms as 
\begin{align*}
u_{\ell,i}=\ULq\PLq u_{q,i}, \quad R_{\ell}=\ULo\PLo R_{q}, \quad 0\leqslant i\leqslant 5,
\end{align*}
which can be defined on $\cI^{q,i}+3\ell_{q,i}\subset\cI^{q,i-1}$ by the selection of sufficiently large $a$. The following estimates are direct conclusions.
\begin{pp}\label{est on Mollification}
For any $0<\beta<\frac{1}{6}$ and $b>1+6\beta,$ we can find $a^*_1=a^*_1(\beta,b, \gamma,M )>0$ satisfying that if $a>a^*_1,$ the following properties hold,
\begin{align}
	\|\partial_t^ru_{\ell,i}\|_N&\lesssim \varepsilon,&& N+r\leqslant 1,&& 0\leqslant i\leqslant 5, \label{est on u_li 0}\\
	\|\partial_t^ru_{\ell,i}\|_N&\lesssim_{N,r}\ell_{q,i}^{2-N-r}M\lambda_{q,i}\delta_{q,i}^{\frac{1}{2}},&& N+r\geqslant2,&& 0\leqslant i\leqslant 5, \label{est on u_li 1}\\
	\|\partial_t^rR_{\ell}\|_N&\lesssim_{N,r}\ell_{q,0}^{1-N-r}\lambda_{q}^{1-2\gamma }\delta_{q+1},&& N+r\geqslant1, \label{est on R_li}\\
	\|\partial_t^r(u_{q,i}-u_{\ell,i})\|_N&\lesssim_{N,r}\ell_{q,i}^{3-N-r}M\lambda_{q,i}^2\delta_{q,i}^{\frac{1}{2}},&& 0\leqslant N+r\leqslant3,&&0\leqslant i\leqslant 5\label{est on u_qi-u_li}, \\
	\|\partial_t^r(R_{q}-R_{\ell})\|_N&\lesssim_{N,r}\ell_{q,0}^{2-N-r}\lambda_{q}^{2-2\gamma }\delta_{q+1},&& 0\leqslant N+r\leqslant2\label{est on R_qi-R_li},
\end{align}
where
$\|\cdot\|_N=\|\cdot\|_{C^0(\cI_{3\ell_{q,i}}^{q,i};C^N(\T^3))}$, and 
\begin{align}
	\delta_{q,i}=\left\lbrace\begin{aligned}
		&\delta_{q},&& i=0,\\
		&\delta_{q+1},&& i=1,\cdots,5.\\
	\end{aligned}\right.
\end{align}
\begin{proof}
Notice that for $0<\beta<\frac{1}{6}$ and $b>1+6\beta,$ we could find $a^*_1=a^*_1(\beta,b, \gamma,M )$ such that for any $a>a^*_1,$
\begin{align}
	\lambda_{q}\delta_{q}^{\frac{1}{2}}\leqslant\lambda_{q,1}\delta_{q+1}^{\frac{1}{2}}, \quad \tau_{q,i}+3\ell_{q,i}\leqslant \tau_{q,i-1}, \quad \ell_{q,i}M\lambda_{q,i}\leqslant 1, \quad  \text{and}\quad \lambda_{q}^{-2\gamma }<r_0,\quad 0\leqslant i\leqslant 5.	\label{pp of Lambda1}	
\end{align}
By using \eqref{est on u_qi}, \eqref{est on R_qi},  and  the definition of $\PLq$ and $\ULq,$ we obtain \eqref{est on u_li 0}--\eqref{est on R_li}. Next, we calculate
$$
\begin{aligned}
	F-\ULq\PLq F=F-\PLq F+\PLq F-\ULq\PLq F=\PGq F+\UGq\PLq F,
\end{aligned}
$$ 
and use Bernstein's inequality and $|f(t-\tau)-f(t)-\sum_{i=1}^{j-1}\partial_t^if(t)(-\tau)^i|\lesssim\ell^j\|\partial_t^jf\|_{C^0([t-\ell,t+\ell])}$ for $|\tau|\leqslant\ell$ to get
\begin{align}
	\|\PGq F\|_{C^0}\lesssim \ell_{q,i}^j\|\nabla^{j}F\|_{C^0}, \quad
	\|\UGq F\|_{C^0}\lesssim \ell_{q,i}^j\|\partial_{t}^jF\|_{C^0}, \label{Bernstein inequality}
\end{align}
for $\forall F\in  C^j(\cI^{q,i-1}\times\T^3),j=0,1,2, \cdots,n_0+3$. Combining them, we have for $N+r\leqslant \overline{N},$
\begin{equation}
	\begin{aligned}
		\|\partial_{t}^r(F-\ULq\PLq F)\|_{N}&\lesssim\ell_{q,i}^{\overline{N}-N-r}\sum_{N_1+N_2=\overline{N}}\|\partial_{t}^{N_1}\nabla^{N_2}F\|_{C^0(\cI^{q,i-1}\times\T^3)}.
	\end{aligned}\notag
\end{equation}
We can apply it to $u_{q,i}$  and $R_{q,i},$ with $\overline{N}=3$ and $\overline{N}=2,$ to get \eqref{est on u_qi-u_li}--\eqref{est on R_qi-R_li}.
\end{proof}
\end{pp}
\subsection{Cutoffs}
Here, we will give partitions of unity in space $\R^3$ and in time $\R$. We introduce some nonnegative smooth functions $\left\lbrace\chi_\upsilon\right\rbrace_{\upsilon\in\Z^3} $ and $\left\lbrace\theta_s\right\rbrace_{s\in\Z} $ such that
$$\sum_{\upsilon\in\Z^3}\chi_\upsilon^2(x)=1, \quad \forall x\in\R^3, \quad\sum_{s\in\Z}\theta_s^2(t)=1, \quad\forall t\in\R,$$
where $\chi_\upsilon(x)=\chi_0(x-2\pi\upsilon)$ and $\chi_0$ is a nonnegative smooth function supported in $Q(0,5/4\pi)$ satisfying $\chi_0=1$ on $\overline{Q(0,3/4\pi)},$ and $Q(x,r)$ denotes the cube $\left\lbrace y\in\R^3:|y-x|_\infty<r\right\rbrace $. Similarly, $\theta_s(t)=\theta_0(t-s)$ where $\theta_0\in C_c^\infty(\R)$ satisfies $\theta_0=1$ on $[1/4,3/4]$ and $\theta_0=0$ on $(-1/4,5/4)^c$. And then, we  give the cut-off parameters $\tau_{q,i}$ and $\mu_{q,i}$ with $\tau_{q,i}^{-1}>0$ and $\mu_{q,i}^{-1}\in 2\Z_+,$ which are given by 
\begin{align}
	\mu_{q,i}^{-1}=2\lceil(\lambda_{q,i}\lambda_{q,i+1})^{\frac{1}{2}}\delta_{q}^{\frac{1}{4}}/2\rceil, \quad \tau_{q,i}^{-1}=(\lambda_{q,i}\lambda_{q,i+1})^{\frac{1}{2}}\delta_{q}^{\frac{1}{4}},\quad 0\leqslant i\leqslant 5. \label{def of mu tau}
\end{align}
Next, we introduce the following notations
\begin{align*}
\mathscr{I}&:=\left\lbrace(s, \upsilon):(s, \upsilon)\in\Z\times\Z^3\right\rbrace,
\end{align*}
and for $I=(s, \upsilon)\in \mathscr{I},$
\begin{equation}
[I]=[s]+\sum_{i=1}^32^i[\upsilon_i]+1, \quad [j]=\left\{\begin{aligned}
&0, && j\ \text{is odd}, \\
&1, && j\ \text{is even}.
\end{aligned}\right.
\end{equation}
So far, we can define the cutoff functions as follows:
\begin{align}
	\chi_I(x)=\chi_\upsilon(\mu_{q,i}^{-1}x), \quad \theta_I(t)=\theta_s(\tau_{q,i}^{-1}t).
\end{align}
\subsection{Definition of building blocks}\label{Definition of Building blocks}
In this part, we give the building blocks which will be used in the construction of the perturbation. 
\begin{lm}\label{construction of building blocks}
Given $f\in\Z^3$ and $\mu,\lambda+\mu>0,$ we could find a nontrivial planar wave solution to the equation
\begin{align*}
	\partial_{tt}w_{A,f}-\mu\Delta w_{A,f}-(\lambda+\mu)\nabla(\Div w_{A,f})+A(\nabla^2w_{A,f})=0,
\end{align*} 
where $w_{A,f}:[0, \infty)\times\T^3\rightarrow \R^3$ and $(A(\nabla^2w))_p=A^{pm}_{nr}\partial_{nr}w_m$ with constants $A^{pm}_{nr}$.\par 
Especially, the solution could have the following expression
\begin{align}
	w_{A,f}=(f+a_{A,2}f^{\perp}+a_{A,3}\frac{f}{|f|}\times f^{\perp})e^{i\xi_{A,f}},
\end{align}
where $f^{\perp}\in\mathbb{Q}^3$ and satisfies  $f^{\perp}\perp f,$ $ |f^\perp|=|f| ,$ $\xi_{A,f}=f\cdot x-((\lambda+2\mu)|f|^2-c_A)^{\frac{1}{2}}t,$ and $c_A$ is a constant depending on coefficients $A^{pm}_{nr}$ and $|f|$. Moreover, there exists $\varepsilon_0(\lambda, \mu,C)>0$ such that, for any $\varepsilon<\varepsilon_0,$ if  $|A_{nr}^{pm}|_0\leqslant C\varepsilon,$ we have
 \begin{align}
 \|a_{A,2}\|_0+\|a_{A,3}\|_0\lesssim_{\lambda,\mu,C}\varepsilon\leqslant \frac{1}{10}, \quad |c_A|\lesssim_{C}\varepsilon|f|^2.\label{est on other dirct cof}
 \end{align}
\begin{proof}
For convenience, we denote by $f^{(i)}=f,f^{\perp},\frac{f}{|f|}\times f^{\perp},$ for $i=1,2,3,$ and $\tilde{A}^{cs}_{ij}=A(f^{(i)},f^{(j)},f^{(s)},f^{(r)})=A^{pm}_{nr}f^{(i)}_{n}f^{(j)}_{r}f^{(s)}_{m}f^{(c)}_{p}|f|^{-2}$. Next, we define  $w_{A,1}:=fe^{i\xi_{A,f}},$ where $c_A$ will be chosen later. Moreover, we have
\begin{align*}
	\partial_{tt}w_{A,1}-\mu\Delta w_{A,1}-(\lambda+\mu)\nabla(\Div w_{A,1})=c_Aw_{A,1}.
\end{align*}
 Notice that
$
A(\nabla^2w_{A,1})=\sum_{r=1}^3(A^{pm}_{nr}\partial_{nr}(w_{A,1})_mf^{(r)}_p)|f|^{-2}f^{(r)}=-\sum_{r=1}^3\tilde{A}^{r1}_{11}f^{(r)}e^{i\xi_{A,f}},
$
we could calculate
\begin{align*}
	\partial_{tt}w_{A,1}-\mu\Delta w_{A,1}-(\lambda+\mu)\nabla(\Div w_{A,1})+A(\nabla^2w_{A,1})=\left(c_Af^{(1)}-\sum_{r=1}^3\tilde{A}^{r1}_{11}f^{(r)}\right)e^{i\xi_{A,f}}.
\end{align*}
In order to cancel the items on the right hand side, we add $w_{A,2}+w_{A,3}$ to $w_{A,1},$ which is  given by
\begin{align*}
	w_{A,2}+w_{A,3}:=(a_{A,2}f^{(2)}+a_{A,3}f^{(3)})e^{i\xi_{A,f}}.
\end{align*}
Then, we have
\begin{align*}
	&\quad\partial_{tt}(w_{A,2}+w_{A,3})-\mu\Delta (w_{A,2}+w_{A,3})-(\lambda+\mu)\nabla(\Div (w_{A,2}+w_{A,3}))+A(\nabla^2 (w_{A,2}+w_{A,3}))\\
	&=(c_A-(\lambda+\mu)|f|^2)(a_{A,2}f^{(2)}+a_{A,3}f^{(3)})e^{i\xi_{A,f}}-\sum_{r=1}^3(a_{A,2}\tilde{A}^{r2}_{11}f^{(r)}+a_{A,3}\tilde{A}^{r3}_{11}f^{(r)})e^{i\xi_{A,f}}.
\end{align*}
So we  need the following equalities
\begin{equation}
	\left\{
	\begin{aligned}
		&a_{A,1}c_A-\sum_{r=1}^3a_{A,r}\tilde{A}^{1r}_{11}=0, \\
		&a_{A,2}(c_A-(\lambda+\mu)|f|^2)-\sum_{r=1}^3a_{A,r}\tilde{A}^{2r}_{11}=0, \\
		&a_{A,3}(c_A-(\lambda+\mu)|f|^2)-\sum_{r=1}^3a_{A,r}\tilde{A}^{3r}_{11}=0.
	\end{aligned}\right.
\end{equation}
where $a_{A,1}=1$. It can be rewritten as 
\begin{equation}
	\left\{
	\begin{aligned}
		&a_{A,2}\left(-\sum_{r=1}^3a_{A,r}|f|^{-2}\tilde{A}^{1r}_{11}+|f|^{-2}\tilde{A}^{22}_{11}+(\lambda+\mu)\right)+a_{A,3}|f|^{-2}\tilde{A}^{23}_{11}=-|f|^{-2}\tilde{A}^{21}_{11}, \\		&a_{A,2}|f|^{-2}\tilde{A}^{32}_{11}+a_{A,3}\left(-\sum_{r=1}^3a_{A,r}|f|^{-2}\tilde{A}^{1r}_{11}+|f|^{-2}\tilde{A}^{33}_{11}+(\lambda+\mu)\right)=-|f|^{-2}\tilde{A}^{31}_{11}.
	\end{aligned}\right.
\end{equation}
Let $ w_{A,f} = \sum_{r=1}^3 w_{A,r} $. For sufficiently small $\varepsilon_0(\lambda, \mu, C) > 0,$ we have for any $\varepsilon < \varepsilon_0,$  $ |A_{nr}^{pm}|_0 \leqslant C\varepsilon $. Consequently, $ |\tilde{A}^{rs}_{ij}| \leqslant C\varepsilon |f|^2 $.
 By using Lemma \ref{polynomial equations} and choosing parameters as $A_1=E_2=-|f|^{-2}\tilde{A}^{12}_{11},$ $A_2=E_1=-|f|^{-2}\tilde{A}^{13}_{11},$ 
 $B_1=|f|^{-2}\tilde{A}^{32}_{11},$ $B_2=|f|^{-2}\tilde{A}^{23}_{11},$ 
 $C_1=-|f|^{-2}\tilde{A}^{11}_{11}+|f|^{-2}\tilde{A}^{22}_{11}+(\lambda+\mu),$ $C_2=-|f|^{-2}\tilde{A}^{11}_{11}+|f|^{-2}\tilde{A}^{33}_{11}+(\lambda+\mu),$
 $D_1=-|f|^{-2}\tilde{A}^{21}_{11},$  $D_2=-|f|^{-2}\tilde{A}^{31}_{11},$ and $c=\lambda+\mu,$ we could get the solution $a_{A,2}=a_1,a_{A,3}=a_2$ which satisfy \eqref{est on other dirct cof} and complete the proof.
\end{proof}
\end{lm} 
\subsection{Definition of the perturbation}\label{Definition of the perturbation}
Up to now, we are ready to construct the perturbations by using the building blocks defined before. For each $I=(s, \upsilon)\in \mathscr{I},$ we define
\begin{equation}\label{definition of A_I}
	\begin{aligned}
	(A_I)^{pm}_{nr}&=2\delta_{rp}\partial_n(u_{\ell,i})_m(s\tau_{q,i},2\pi\upsilon\mu_{q,i})-\delta_{rp}\partial_n(u_{\ell,i})_m(s\tau_{q,i},2\pi\upsilon\mu_{q,i})-\delta_{nj}\delta_{rj}\partial_p(u_{\ell,i})_m(s\tau_{q,i},2\pi\upsilon\mu_{q,i})\\
	&=\delta_{rp}\partial_n(u_{\ell,i})_m(s\tau_{q,i},2\pi\upsilon\mu_{q,i})-\delta_{nj}\delta_{rj}\partial_p(u_{\ell,i})_m(s\tau_{q,i},2\pi\upsilon\mu_{q,i}).
	\end{aligned}
\end{equation}
For convenience, we denote $ h^{s, \upsilon} = h(s\tau_{q,i}, 2\pi \upsilon \mu_{q,i}) $ for any function $ h \in C^0(\cI^{q,i}_{3\ell_{q,i}} \times \T^3) $. Notice that $ (\partial_j u_{\ell,i})^{s, \upsilon}_k $ is a constant on $\supp\theta_I\times\supp\chi_I=(s\tau_{q,i}-\frac{1}{4}\tau_{q,i},s\tau_{q,i}+\frac{5}{4}\tau_{q,i})\times Q(2\pi \upsilon \mu_{q,i}, \frac{5}{4}\pi\upsilon \mu_{q,i}).$ We have
\begin{align*}
	(A_I(\nabla^2 w))_p=(A_I)^{pm}_{nr}\partial_{nr}w_m
	&=2\partial_{jp}w_k (\partial_ju_{\ell,i})^{s, \upsilon}_k-\partial_{jp}w_k (\partial_ju_{\ell,i})^{s, \upsilon}_k-(\partial_{p}u_{\ell,i})^{s, \upsilon}_k\partial_{jj} w_k\\
	&=\left(\Div(\tr(\nabla w(\nabla u_{\ell,i}^{s, \upsilon})^{\top}+\nabla u_{\ell,i}^{s, \upsilon}(\nabla w)^{\top})\Id-(\nabla w)^{\top}\nabla u_{\ell,i}^{s, \upsilon}-(\nabla u_{\ell,i}^{s, \upsilon})^{\top}\nabla w)\right)_p.
\end{align*}
\begin{Rmk}
We use the notation $\nabla u_{\ell,i}^{s, \upsilon}$ instead of $(\nabla u_{\ell,i})^{s, \upsilon}$ for simplicity.
\end{Rmk}
Let $\xi_{A_I,f_{i+1}}=[I]\left(f_{i+1}\cdot x-((\lambda+2\mu)|f_{i+1}|^2-c_{A_I})^{\frac{1}{2}}t\right)$ and 
\begin{align*}
	w_{A_I,f_{i+1}}&=(f_{i+1}+a_{A_I,2}f_{i+1}^{(2)}+a_{A_I,3}f_{i+1}^{(3)})e^{i\lambda_{q,i+1}\xi_{A_I,f_{i+1}}}\\
	&=(i\lambda_{q,i+1}[I]|f_{i+1}|^2)^{-1}\Div((f_{i+1}+a_{A_I,2}f_{i+1}^{(2)}+a_{A_I,3}f_{i+1}^{(3)})\otimes f_{i+1}e^{i\lambda_{q,i+1}\xi_{A_I,f_{i+1}}}),
\end{align*}
where $c_{A_I},$ $a_{A_I,2}$ and $a_{A_I,3}$ can be chosen by Lemma \ref{construction of building blocks}. Then, we could calculate
\begin{align}
	&\partial_{tt}w_{A_I,f_{i+1}}-\mu\Delta w_{A_I,f_{i+1}}-(\lambda+\mu)\nabla(\Div w_{A_I,f_{i+1}})\nonumber\\
	&\quad+\Div\left(\tr(2\nabla u_{\ell,i}^{s, \upsilon}(\nabla w_{A_I,f_{i+1}})^{\top})\Id-(\nabla w_{A_I,f_{i+1}})^{\top}\nabla u_{\ell,i}^{s, \upsilon}-(\nabla u_{\ell,i}^{s, \upsilon})^{\top}\nabla w_{A_I,f_{i+1}}\right)=0.\label{prop of w_A}
\end{align} 
By using \eqref{est on R_q} and \eqref{pp of Lambda1}, we have $\|\delta_{q+1}^{-1}R_\ell\|_0< \lambda_{q}^{-2\gamma}\leqslant r_0$.
Then, we could define the weight coefficient $d_{q,i+1}$ which depends on $R_{\ell}$ and $I$ as
\begin{align*}
	d_{q,i+1}(t,x):=\delta_{q+1}^{\frac{1}{2}}\Gamma_{f_{i+1}}(\Id-\delta_{q+1}^{-1}R_{\ell}).
\end{align*}
We first define the main part of the perturbation $\tu_{q,i+1,p}$ as
\begin{align}
	\tu_{q,i+1,p}
	&=\sum_I\frac{\theta_I\chi_Id_{q,i+1}\tilde{f}_{A_I}}{\sqrt{2}\lambda_{q,i+1}[I]|\tilde{f}_{A_I}||f_{i+1}|}(e^{i\lambda_{q,i+1}\xi_{A_I,f_{i+1}}}+e^{-i\lambda_{q,i+1}\xi_{A_I,f_{i+1}}})\nonumber\\
	&=\sum_{I}\frac{1}{\lambda_{q,i+1}[I]}u_{q,i+1,I}(e^{i\lambda_{q,i+1}\xi_{A_I,f_{i+1}}}+e^{-i\lambda_{q,i+1}\xi_{A_I,f_{i+1}}}), \label{def of tu_q i+1,p}
\end{align}
where $\sum_I=\sum_{s}\sum_{\upsilon}$ and
\begin{align}
	\tilde{f}_{A_I}&=f_{i+1}+a_{A_I,2}f_{i+1}^{(2)}+a_{A_I,3}f_{i+1}^{(3)}, \quad\gamma_{q,i+1,I}=\frac{\theta_I\chi_Id_{q,i+1}}{\sqrt{2}|\tilde{f}_{A_I}||f_{i+1}|}, \quad u_{q,i+1,I}=\gamma_{q,i+1,I}\tilde{f}_{A_I}.\label{def of cof}
\end{align}
Moreover, $\supp u_{q,i+1,I}\bigcap\supp u_{q,i+1,J}=\emptyset,$ if $1<\|I-J\|:=\max\{|s(I)-s(J)|,\max_{1\leqslant i\leqslant 3}\{|\upsilon_{i}(I)-\upsilon_{i}(J)|\}\}$. The main part of the perturbation can be also written as
\begin{align*}
	\tu_{q,i+1,p}&=\sum_{I}\frac{1}{\lambda_{q,i+1}[I]}\gamma_{q,i+1,I}(w_{A_I,f_{i+1}}+\overline{w}_{A_I,f_{i+1}}).
\end{align*} \par 
If we add the perturbation $\tilde{u}_{q,i+1}$ to $u_{q,i},$  defining $u_{q,i+1}=u_{q,i}+\tu_{q,i+1},$ we want to find $R_{q,i+1}$ such that
\begin{align*}
	\partial_{tt}u_{q,i+1}-\mu\Delta u_{q,i+1}-(\lambda+\mu)\nabla\Div u_{q,i+1}+\Div(\tr(\nabla u_{q,i+1}(\nabla u_{q,i+1})^{\top})\Id-(\nabla u_{q,i+1})^\top\nabla u_{q,i+1})=\Div(R_{q,i+1}-c_q\Id).
\end{align*}
Integrating the left-hand side of the equations with respect to $x,$ we could get 
\begin{align*}
	&\int_{\T^3}\left(\partial_{tt}u_{q,i+1}-\mu\Delta u_{q,i+1}-(\lambda+\mu)\nabla\Div u_{q,i+1}+\Div(\tr(\nabla u_{q,i+1}(\nabla u_{q,i+1})^{\top})\Id-(\nabla u_{q,i+1})^\top\nabla u_{q,i+1})\right)\rd x\\
	&\quad=\int_{\T^3}\Div(R_q-c_q\Id)\rd x+\int_{\T^3}(\partial_{tt}u_{q,i+1}-\partial_{tt}u_{q,i})\rd x=\int_{\T^3}\partial_{tt}\tu_{q,i+1}\rd x=0.
\end{align*}
So we need to add a time correction term $\tu_{q,i+1,t}=g(t)$ which is a function of $t$ and satisfies that 
\begin{align}
\int_{\T^3}\partial_{tt}\tu_{q,i+1,p}(t,x)+\partial_{tt}g(t)\rd x=0,\quad t\in\cI^{q,i}_{3\ell_{q,i}}.
\end{align}
Here, we can simply take
\begin{align}
	\tu_{q,i+1,t}=g(t)=-\int_{\T^3}\tu_{q,i+1,p}(t,x)\rd x.
\end{align}
So we could define the perturbation
\begin{align}
	\tu_{q,i+1}&=\tu_{q,i+1,p}+\tu_{q,i+1,t}.
\end{align}
\begin{Rmk}
Here, we have two methods to make $\int_{\T^3}\partial_{tt}\tilde{u}_{q,i+1}\, \mathrm{d}x = 0$. One method, similar to constructing perturbations in the Euler equations, is to construct a perturbation in divergence form or curl form. The second method, as shown here, is to add a perturbation that depends only on time. The advantage of the second method is that it does not affect the form of $\nabla \tilde{u}_{q,i+1}$.
\end{Rmk}

Moreover, we could calculate
\begin{align*}
	\nabla\tu_{q,i+1}&=\underbrace{i\sum_{I}u_{q,i+1,I}\otimes f_{i+1}(e^{i\lambda_{q,i+1}\xi_{A_I,f_{i+1}}}-e^{-i\lambda_{q,i+1}\xi_{A_I,f_{i+1}}})}_{\tilde{w}_{q,i+1,p}}+\underbrace{\sum_{I}\frac{1}{\lambda_{q,i+1}[I]}\nabla u_{q,i+1,I}(e^{i\lambda_{q,i+1}\xi_{A_I,f_{i+1}}}+e^{-i\lambda_{q,i+1}\xi_{A_I,f_{i+1}}})}_{\tilde{w}_{q,i+1,c}}.
\end{align*}
Finally, we will give the following estimates on the perturbation.
\begin{pp}\label{est on perturbation}
	There exists $\varepsilon_1=\varepsilon_1(\lambda, \mu)$ such that for any $\varepsilon<\varepsilon_1,$ we have the following estimates:
	\begin{align}
	\|\partial_t^ru_{q,i+1,I}\|_N+	\|\partial_t^r\gamma_{q,i+1,I}\|_N&\lesssim_{\lambda,\mu,N,r} \tau_{q,i}^{-r}\mu_{q,i}^{-N}\delta_{q+1}^{\frac{1}{2}},&& 0\leqslant N+r, \label{est on u_I, gamma_I}\\
	\|\partial_t^rw_{A_I,f_{i+1}}\|_{N}&\lesssim_{\lambda,\mu} \lambda_{q,i+1}^{N+r},&& 0\leqslant N+r\leqslant 3, \label{est on w_I}\\
	\|\partial_t^r\tilde{u}_{q,i+1,t}\|_N&\lesssim_{\lambda,\mu,n_0}\lambda_{q,i+1}^{r-2} \delta_{q+1}^{\frac{1}{2}},&& 0\leqslant N+r\leqslant 3, \label{est on tu_q t}\\
	\|\partial_t^r\tilde{u}_{q,i+1,p}\|_N+\|\partial_t^r\tilde{u}_{q,i+1}\|_N&\lesssim_{\lambda,\mu,n_0} \lambda_{q,i+1}^{N+r-1}\delta_{q+1}^{\frac{1}{2}},&& 0\leqslant N+r\leqslant 3, \label{est on tu_q}\\	
	\|\partial_t^r\tilde{w}_{q,i+1,p}\|_N&\lesssim_{\lambda,\mu} \lambda_{q,i+1}^{N+r}\delta_{q+1}^{\frac{1}{2}},&& 0\leqslant N+r\leqslant 3, \label{est on nabla tu_q p}\\
	\|\partial_t^r\tilde{w}_{q,i+1,c}\|_N&\lesssim_{\lambda,\mu} (\lambda_{q,i+1}\mu_{q,i})^{-1}\lambda_{q,i+1}^{N+r}\delta_{q+1}^{\frac{1}{2}},&& 0\leqslant N+r\leqslant 3, \label{est on nabla tu_q c}
	\end{align}
where $0\leqslant i\leqslant 5$ and
$\|\cdot\|_N=\|\cdot\|_{C^0(\cI_{3\ell_{q,i}}^{q,i};C^N(\T^3))}$. Moreover, the implicit constants in \eqref{est on u_I, gamma_I}--\eqref{est on nabla tu_q c} can be chosen to be independent of $M$. 
\begin{proof}
	By the definition of $A_I$ in \eqref{definition of A_I}, we could know  $(A_I)^{pm}_{nr}$ are constants in each $I=(s, \upsilon)\in \mathscr{I}$ and
	\begin{align}
		|(A_I)^{pm}_{nr}|\lesssim\|\nabla u_{\ell,i}\|_0\lesssim\varepsilon, \label{est on A_I}
	\end{align}
	where the implicit constants does not depend on $M$. By using Lemma \ref{polynomial equations} and Lemma \ref{construction of building blocks}, we could find $\varepsilon_1=\varepsilon_1(\lambda, \mu)$ such that  for any $\varepsilon<\varepsilon_1(\lambda, \mu),$ the building block $w_{A_I,f_{i+1}}$ satisfies  
	\begin{align}
		\|a_{A_I,2}\|_0+\|a_{A_I,3}\|_0\lesssim_{\lambda,\mu}\varepsilon\leqslant \frac{1}{10}, \quad |c_{A_I}|\lesssim\varepsilon|f|^2. \label{est on a_i and c}
	\end{align}
	By using \eqref{def of mu tau}, \eqref{def of cof}, \eqref{est on A_I}, and \eqref{est on a_i and c}, we could easily obtain \eqref{est on u_I, gamma_I}.\par 
	To get \eqref{est on tu_q t}, by using  $\nabla\xi_{A_I,f_{i+1}}=[I]f_{i+1},$  $\supp u_{q,i+1,I}\bigcap\supp u_{q,i+1,J}=\emptyset$ for $\|I-J\|>1,$ \eqref{def of tu_q i+1,p}, and Lemma \ref{est on int operator}, we could obtain for $0\leqslant r\leqslant3,$
	\begin{align*}
		\|\partial_t^r\tilde{u}_{q,i+1,t}\|_N&=\left\|\int_{\T^3}\partial_t^r\tu_{q,i+1,p}(\cdot,x)\rd x\right\|_{C^{0}(\cI_{3\ell_{q,i}}^{q,i})}\\
		&\lesssim
		_{\lambda,\mu,n_0}\frac{1}{\lambda_{q,i+1}}\sum_{I}\sum_{r_1+r_2=r}\left(\frac{\lambda_{q,i+1}^{r_1}\|\partial_t^{r_2}u_{q,i+1,I}\|_{n_0+1}+\lambda_{q,i+1}^{r_1}\|\partial_t^{r_2}u_{q,i+1,I}\|_{0}\|\nabla\xi_{A_I,f_{i+1}}\|_N}{(\lambda_{q,i+1}|f|)^{n_0+1}}\right)\\
		&\lesssim_{\lambda,\mu,n_0}(\lambda_{q,i+1}\mu_{q,i})^{-(n_0+1)}\lambda_{q,i+1}^{r-1}\delta_{q+1}^{\frac{1}{2}}\lesssim\lambda_{q,i+1}^{r-2}\delta_{q+1}^{\frac{1}{2}},
	\end{align*} 
	where the last inequality comes from \eqref{prop of tn_0}. \eqref{est on w_I}, \eqref{est on tu_q},  \eqref{est on nabla tu_q p}, and \eqref{est on nabla tu_q c} are immediate consequence. Moreover, the implicit constants in \eqref{est on u_I, gamma_I}--\eqref{est on nabla tu_q c} can be chosen to be independent of $M,$ because $\mu_{q,i},$ $\tau_{q,i} $ and the implicit constants in \eqref{est on A_I} and \eqref{est on a_i and c} are independent of $M$.
	
\end{proof}
\end{pp}
\section{Definition of the new Reynolds error}\label{Definition of the new errors}  
In the previous section, we have constructed the perturbation.  Here, we  will add the perturbation $\tu_{q,i+1}$ to $u_{q,i},$ defining $u_{q,i+1}=u_{q,i}+\tu_{q,i+1},$ then we could define the new Reynolds error $R_{q,i+1}$.
\subsection{Definition of the new Reynolds error}
We will use the inverse divergence operator $\cR$ to define the new error $R_{q,i+1}$. More details about the inverse divergence operator can be found in Appendix \ref{Inverse divergence operator}. $\Div(R_{q,i+1})$ can be written as
\begin{align*}
	&\quad\Div(R_{q,i+1})\\
	&=\underbrace{\partial_{tt}\tu_{q,i+1}-\mu\Delta \tu_{q,i+1}-(\lambda+\mu)\nabla\Div \tu_{q,i+1}+\Div(\tr(2\nabla u_{\ell,i}(\nabla \tu_{q,i+1})^{\top})\Id-(\nabla \tu_{q,i+1})^\top\nabla u_{\ell,i}-(\nabla u_{\ell,i})^\top\nabla \tu_{q,i+1})}_{\Div R_{L}=\Div R_{L1}+\Div R_{L2}}\\
	&\quad+\underbrace{\Div(2\tr(\nabla (u_{q,i}-u_{\ell,i})(\nabla \tu_{q,i+1})^{\top})\Id-(\nabla \tu_{q,i+1})^\top\nabla (u_{q,i}-u_{\ell,i})-(\nabla (u_{q,i}-u_{\ell,i}))^\top\nabla \tu_{q,i+1})}_{\Div R_{M}}\\
	&\quad+\underbrace{\Div\left(\tr(\nabla \tu_{q,i+1}(\nabla \tu_{q,i+1})^{\top})\Id-(\nabla \tu_{q,i+1})^\top\nabla \tu_{q,i+1}-(\delta_{q+1}^{\frac{1}{2}}\Gamma_{f_{i+1}}(\Id-\delta_{q+1}^{-1}R_{\ell}))^2\left(\Id-\frac{f_{i+1}}{|f_{i+1}|}\otimes \frac{f_{i+1}}{|f_{i+1}|}\right)\right)}_{\Div R_O=\Div R_{O1}+\Div R_{O2}}\\
	&\quad+\Div \left(R_{q,i}+(\delta_{q+1}^{\frac{1}{2}}\Gamma_{f_{i+1}}(\Id-\delta_{q+1}^{-1}R_{\ell}))^2\left(\Id-\frac{f_{i+1}}{|f_{i+1}|}\otimes \frac{f_{i+1}}{|f_{i+1}|}\right)\right),
\end{align*}
and then we define the new Reynolds error
\begin{align}
	R_{q,i+1}&=R_{q,i}+(\delta_{q+1}^{\frac{1}{2}}\Gamma_{f_{i+1}}(\Id-\delta_{q+1}^{-1}R_{\ell}))^2\left(\Id-\frac{f_{i+1}}{|f_{i+1}|}\otimes \frac{f_{i+1}}{|f_{i+1}|}\right)+\delta R_{q,i+1}, \\
	\delta R_{q,i+1}&=R_{L}+R_{M}+R_{O}, \label{split of R_q i+1}
\end{align}
where
\begin{align}
	R_{M}&=\tr(2\nabla (u_{q,i}-u_{\ell,i})(\nabla \tu_{q,i+1})^{\top})\Id-(\nabla \tu_{q,i+1})^\top\nabla (u_{q,i}-u_{\ell,i})-(\nabla (u_{q,i}-u_{\ell,i}))^\top\nabla \tu_{q,i+1}, \label{def of R_M}\\
	R_{O1}&=\cR\Div\left(\tr(\tilde{w}_{q,i+1,p}(\tilde{w}_{q,i+1,p})^{\top})\Id-(\tilde{w}_{q,i+1,p})^\top\tilde{w}_{q,i+1,p}-(\delta_{q+1}^{\frac{1}{2}}\Gamma_{f_{i+1}}(\Id-\delta_{q+1}^{-1}R_{\ell}))^2\left(\Id-\frac{f_{i+1}}{|f_{i+1}|}\otimes \frac{f_{i+1}}{|f_{i+1}|}\right)\right)\nonumber\\
	&=\cR\Div\left(\tr(\tilde{w}_{q,i+1,p}(\tilde{w}_{q,i+1,p})^{\top})\Id-(\tilde{w}_{q,i+1,p})^\top\tilde{w}_{q,i+1,p}-2|f_{i+1}|^2\sum_I|u_{q,i+1,I}|^2\left(\Id-\frac{f_{i+1}}{|f_{i+1}|}\otimes \frac{f_{i+1}}{|f_{i+1}|}\right)\right), \label{def of R_O1}\\
	R_{O2}&=\tr(\tilde{w}_{q,i+1,p}(\tilde{w}_{q,i+1,c})^{\top}+\tilde{w}_{q,i+1,c}(\tilde{w}_{q,i+1,p})^{\top}+\tilde{w}_{q,i+1,c}(\tilde{w}_{q,i+1,c})^{\top})\Id\nonumber\\
	&\quad-(\tilde{w}_{q,i+1,c})^\top\tilde{w}_{q,i+1,p}-(\tilde{w}_{q,i+1,p})^\top\tilde{w}_{q,i+1,c}-(\tilde{w}_{q,i+1,c})^\top\tilde{w}_{q,i+1,c},\label{def of R_O2} \\
	R_{O}&:=R_{O1}+R_{O2}\label{def of R_O}.
\end{align}
The definition of $R_{L}$ is a little different. Notice that
\begin{align*}
	\Div R_{L}&=\partial_{tt}\tu_{q,i+1}-\mu\Delta \tu_{q,i+1}-(\lambda+\mu)\nabla\Div \tu_{q,i+1}+\Div(\tr(2\nabla u_{\ell,i}(\nabla \tu_{q,i+1})^{\top})\Id-(\nabla \tu_{q,i+1})^\top\nabla u_{\ell,i}-(\nabla u_{\ell,i})^\top\nabla \tu_{q,i+1}).
\end{align*} 
By using \eqref{prop of w_A}, we could calculate
\begin{align*}
	&\quad\partial_{tt}\tu_{q,i+1,p}-\mu\Delta \tu_{q,i+1,p}-(\lambda+\mu)\nabla\Div \tu_{q,i+1,p}\\
	&=\sum_{I}\frac{1}{\lambda_{q,i+1}[I]}\left((\partial_{tt}\gamma_{q,i+1,I}-\mu\Delta\gamma_{q,i+1,I})(w_{A_I,f_{i+1}}+\overline{w}_{A_I,f_{i+1}})-(\lambda+\mu)\nabla((w_{A_I,f_{i+1}}+\overline{w}_{A_I,f_{i+1}})\cdot\nabla\gamma_{q,i+1,I})\right)\\
	&\quad+\sum_{I}\frac{1}{\lambda_{q,i+1}[I]}\left(2\partial_t\gamma_{q,i+1,I}\partial_t(w_{A_I,f_{i+1}}+\overline{w}_{A_I,f_{i+1}})-2\mu\nabla(w_{A_I,f_{i+1}}+\overline{w}_{A_I,f_{i+1}})\nabla\gamma_{q,i+1,I}\right) \\
	&\quad-\sum_{I}\frac{\lambda+\mu}{\lambda_{q,i+1}[I]}\left(\Div(w_{A_I,f_{i+1}}+\overline{w}_{A_I,f_{i+1}})\nabla\gamma_{q,i+1,I}\right) -\sum_{I}\frac{\gamma_{q,i+1,I}}{\lambda_{q,i+1}[I]}\Div(\tr(2\nabla u_{\ell,i}^{s, \upsilon}(\nabla (w_{A_I,f_{i+1}}+\overline{w}_{A_I,f_{i+1}}))^{\top})\Id)\\
	&\quad+\sum_{I}\frac{\gamma_{q,i+1,I}}{\lambda_{q,i+1}[I]}\Div((\nabla (w_{A_I,f_{i+1}}+\overline{w}_{A_I,f_{i+1}}))^{\top}\nabla u_{\ell,i}^{s, \upsilon}+(\nabla u_{\ell,i}^{s, \upsilon})^{\top}\nabla (w_{A_I,f_{i+1}}+\overline{w}_{A_I,f_{i+1}})).
\end{align*} 
Then we could divide the linear error $R_L$ into two parts:
\begin{align} 		  R_{L1}&:=\cR\Big(\sum_{I}\frac{1}{\lambda_{q,i+1}[I]}\left((\partial_{tt}\gamma_{q,i+1,I}-\mu\Delta\gamma_{q,i+1,I})(w_{A_I,f_{i+1}}+\overline{w}_{A_I,f_{i+1}})\right)\nonumber\\
	&\quad+\sum_{I}\frac{1}{\lambda_{q,i+1}[I]}\left(2\partial_t\gamma_{q,i+1,I}\partial_t(w_{A_I,f_{i+1}}+\overline{w}_{A_I,f_{i+1}})-2\mu\nabla(w_{A_I,f_{i+1}}+\overline{w}_{A_I,f_{i+1}})\nabla\gamma_{q,i+1,I}\right) \nonumber\\
	&\quad-\sum_{I}\frac{\lambda+\mu}{\lambda_{q,i+1}[I]}\left(\nabla((w_{A_I,f_{i+1}}+\overline{w}_{A_I,f_{i+1}})\cdot\nabla\gamma_{q,i+1,I})+\Div(w_{A_I,f_{i+1}}+\overline{w}_{A_I,f_{i+1}})\nabla\gamma_{q,i+1,I}\right)\label{def of R_L1}\\
	&\quad+\sum_{I}\frac{1}{\lambda_{q,i+1}[I]}\tr(2\nabla u_{\ell,i}^{s, \upsilon} (\nabla(w_{A_I,f_{i+1}}+\overline{w}_{A_I,f_{i+1}}))^{\top})\nabla\gamma_{q,i+1,I}\nonumber\\
	&\quad-\sum_{I}\frac{1}{\lambda_{q,i+1}[I]}\left((\nabla (w_{A_I,f_{i+1}}+\overline{w}_{A_I,f_{i+1}}))^{\top}\nabla u_{\ell,i}^{s, \upsilon}+(\nabla u_{\ell,i}^{s, \upsilon})^{\top}\nabla(w_{A_I,f_{i+1}}+\overline{w}_{A_I,f_{i+1}})\right)\nabla\gamma_{q,i+1,I}+\partial_{tt}\tu_{q,i+1,t}\Big),\nonumber\\
	R_{L2}
	&:=\tr(2\nabla u_{\ell,i}(\nabla \tu_{q,i+1})^{\top})\Id-(\nabla \tu_{q,i+1})^\top\nabla u_{\ell,i}-(\nabla u_{\ell,i})^\top\nabla \tu_{q,i+1}\nonumber\\
	&\quad-\sum_{I}\frac{\gamma_{q,i+1,I}}{\lambda_{q,i+1}[I]}\tr(2\nabla u_{\ell,i}^{s, \upsilon}(\nabla (w_{A_I,f_{i+1}}+\overline{w}_{A_I,f_{i+1}}))^{\top})\Id\label{def of R_L2}\\
	&\quad+\sum_{I}\frac{\gamma_{q,i+1,I}}{\lambda_{q,i+1}[I]}(\nabla (w_{A_I,f_{i+1}}+\overline{w}_{A_I,f_{i+1}}))^{\top}\nabla u_{\ell,i}^{s, \upsilon}+(\nabla u_{\ell,i}^{s, \upsilon})^{\top}\nabla (w_{A_I,f_{i+1}}+\overline{w}_{A_I,f_{i+1}}), \nonumber\\
	R_{L}&:=R_{L1}+R_{L2}\label{def of R_L}.
\end{align}
\section{Estimates on the new Reynolds error}\label{Estimates on the new Reynolds error}
In this section, we will give the estimates on the new Reynolds stress $R_{q+1}$ and its derivative. For convenience, if the implicit constants in the estimates in this section depend on $\lambda,$ $\mu,$ and $M,$ we will not write it out.  For the remaining sections, we set $\|\cdot\|_N =\|\cdot\|_{C^0(\cI^{q,i}; C^N(\T^3))}$ and fix $n_0 =\left\lceil\frac{12b(1+2\gamma)}{b-1+6\beta}\right\rceil$ so that
\begin{equation}\label{prop of tn_0}
	({\lambda_{q,i+1}\mu_{q,i}})^{-n_0}\mu_{q,i}^{-2\gamma}\lesssim	({\lambda_{q,i+1}\mu_{q,i}})^{-n_0}\lambda_{q,i+1}^{2\gamma}\lesssim \lambda_{q,i+1}^{-1}, \quad 0\leqslant i\leqslant 5.
\end{equation}
\begin{pp}\label{prop of Reynolds error}
	For any $0<\beta<\frac{1}{60},$ let the parameters $a^*_1$ and $\varepsilon_1$ be as in the statement of  Propositions \ref{est on Mollification} and \ref{est on perturbation}. Then, if  $\bar{b}(\beta)=\frac{1+12 \beta-36 \gamma}{48 \beta}>1+6\beta$ and $\varepsilon<\varepsilon_1,$ we can find  $a^*_0=a^*_0(\beta,b, \gamma ,M)\geqslant a^*_1$ such that for $b=\overline{b}$ and  any $a>a^*_0,$ we have  for $0\leqslant i\leqslant 5,$
	\begin{align}
		\|\partial_t^r\delta R_{q,i+1}\|_{N}&\leqslant C_{\lambda,\mu,M}\lambda_{q,i+1}^{N+r+\gamma} \cdot
		{\lambda_{q,i}^\frac{1}{2}}\lambda_{q,i+1}^{-\frac{1}{2}} \delta_q^\frac{1}{4}\delta_{q+1}^\frac{1}{2} 
		\leqslant \frac{1}{12} \lambda_{q,i+1}^{N+r-2\gamma} \delta_{q+2},\label{est on delta R_q i+1},&&0\leqslant N+r\leqslant2,\\
		\|\partial_t^r( R_{q}-R_{\ell})\|_{N}&\lesssim\ell_{q,0}^{2-N-r}\lambda_{q}^{2-2\gamma }\delta_{q+1}\leqslant\frac{1}{12} \lambda_{q,i+1}^{N+r-2\gamma} \delta_{q+2},&&0\leqslant N+r\leqslant2, \label{est on R_q-R_l}
	\end{align}
	where $C_{\lambda,\mu,M}$ depends only upon $\lambda,$ $\mu,$ and $M$ in Propositions \ref{Inductive proposition} and \ref{Bifurcating inductive proposition}.
\end{pp}
We will consider \eqref{split of R_q i+1} and  estimate the separate terms $   R_{L},$ $ R_{M},$ and $R_{O}$. For the errors $R_{L2},$ $R_{O2},$ and $R_{M},$ we use a direct estimate. For $R_{L1}$ and $R_{O1},$ we use \eqref{est on R ae} in Lemma \ref{est on int operator}. Remark that 
\begin{equation}\label{rel.par}
\begin{aligned}
	\ell_{q,i}^2\lambda_{q,i}^2 \delta_{q+1}^{\frac{1}{2}}\delta_{q,i}^{\frac{1}{2}}+(\lambda_{q,i+1} \mu_{q,i})^{-1}\delta_{q+1}^\frac{1}{2}+(\tau_{q,i}+\mu_{q,i})\lambda_{q,i}\delta_{q+1}^{\frac{1}{2}}\delta_{q,i}^{\frac{1}{2}}&\lesssim  {\lambda_{q,i}^\frac{1}{2}}{\lambda_{q,i+1}^{-\frac{1}{2}}} \delta_q^\frac{1}{4}\delta_{q+1}^\frac{1}{2},&& 0\leqslant i\leqslant 5.
\end{aligned}
\end{equation}
\subsection{Estimates on the linear error}
Recalling  the definition of $R_{L1},$ we could calculate
\begin{align*}
 R_{L1}&=\cR\Big(\sum_{I}\left(\frac{\partial_{tt}u_{q,i+1,I}-\mu\Delta u_{q,i+1,I}-(\lambda+\mu)\nabla\Div u_{q,i+1,I}}{\lambda_{q,i+1}[I]}\right)(e^{i\lambda_{q,i+1}\xi_{A_I,f_{i+1}}}+e^{-i\lambda_{q,i+1}\xi_{A_I,f_{i+1}}})\\
	&\quad-\sum_{I}2i\left(\partial_{t}u_{q,i+1,I}((\lambda+2\mu)|f_{i+1}|^2-c_{A_I})^{\frac{1}{2}}+\mu (f_{i+1}\cdot\nabla) u_{q,i+1,I}\right)(e^{i\lambda_{q,i+1}\xi_{A_I,f_{i+1}}}-e^{-i\lambda_{q,i+1}\xi_{A_I,f_{i+1}}})\\
	&\quad-\sum_{I}(\lambda+\mu)i\left(\Div u_{q,i+1,I}f_{i+1}+\nabla( f_{i+1}\cdot u_{q,i+1,I})\right)(e^{i\lambda_{q,i+1}\xi_{A_I,f_{i+1}}}-e^{-i\lambda_{q,i+1}\xi_{A_I,f_{i+1}}})\\
	&\quad+\sum_{I}2i(\tilde{f}_{A_I}\cdot(\nabla u_{\ell,i}^{s, \upsilon}f_{i+1}))\nabla\gamma_{q,i+1,I}(e^{i\lambda_{q,i+1}\xi_{A_I,f_{i+1}}}-e^{-i\lambda_{q,i+1}\xi_{A_I,f_{i+1}}})\\
	&\quad-\sum_{I}i(f_{i+1}\otimes((\nabla u_{\ell,i}^{s, \upsilon})^{\top}\tilde{f}_{A_I})+((\nabla u_{\ell,i}^{s, \upsilon})^{\top}f_{A_I})\otimes f_{i+1})\nabla\gamma_{q,i+1,I}(e^{i\lambda_{q,i+1}\xi_{A_I,f_{i+1}}}-e^{-i\lambda_{q,i+1}\xi_{A_I,f_{i+1}}})+\partial_{tt}\tu_{q,i+1,t}\Big),
\end{align*}
and then, by using  $\nabla\xi_{A_I,f_{i+1}}=[I]f_{i+1},$  $\supp u_{q,i+1,I}\bigcap\supp u_{q,i+1,J}=\emptyset$ for $\|I-J\|>1,$ and Lemma \ref{est on int operator}, we have
\begin{align*}
\|\partial_t^rR_{L1}\|_{N}\leqslant\|\partial_t^rR_{L1}\|_{N+\gamma}
&\lesssim\sum_{I}\lambda_{q,i+1}^{N+r}\left(\frac{\|\partial_{tt}u_{q,i+1,I}\|_{0}+\|u_{q,i+1,I}\|_{2}}{\lambda_{q,i+1}^{2-\gamma}}+\frac{\|\partial_{tt}u_{q,i+1,I}\|_{n_0+\gamma}+\|u_{q,i+1,I}\|_{n_0+2+\gamma}}{\lambda_{q,i+1}^{n_0+1-\gamma}}\right)\\
&\quad+\sum_{I}\lambda_{q,i+1}^{N+r}\left(\frac{\|\partial_{t}u_{q,i+1,I}\|_{0}+\|u_{q,i+1,I}\|_{1}}{\lambda_{q,i+1}^{1-\gamma}}+\frac{\|\partial_{t}u_{q,i+1,I}\|_{n_0+\gamma}+\|u_{q,i+1,I}\|_{n_0+1+\gamma}}{\lambda_{q,i+1}^{n_0-\gamma}}\right)\\
&\quad+\sum_{I}\lambda_{q,i+1}^{N+r}\left(\frac{\|\nabla u_{\ell,i}^{s, \upsilon}\|_0\|\nabla \gamma_{q,i+1,I}\|_{0}}{\lambda_{q,i+1}^{1-\gamma}}+\frac{\|\nabla u_{\ell,i}^{s, \upsilon}\|_{0}\|\nabla \gamma_{q,i+1,I}\|_{n_0+\gamma}}{\lambda_{q,i+1}^{n_0-\gamma}}\right)\\
&\lesssim\lambda_{q,i+1}^{N+r}\left(\frac{\delta_{q+1}^{\frac{1}{2}}}{\lambda_{q,i+1}^{2-\gamma}\mu_{q,i}^{2}}+\frac{\delta_{q+1}^{\frac{1}{2}}}{\lambda_{q,i+1}^{n_0+1-\gamma}\mu_{q,i}^{n_0+2+\gamma}}+\frac{\delta_{q+1}^{\frac{1}{2}}}{\lambda_{q,i+1}^{1-\gamma}\mu_{q,i}}+\frac{\delta_{q+1}^{\frac{1}{2}}}{\lambda_{q,i+1}^{n_0-\gamma}\mu_{q,i}^{n_0+1+\gamma}}\right)\\
&\lesssim\lambda_{q,i+1}^{N+r}\frac{\delta_{q+1}^{\frac{1}{2}}}{\lambda_{q,i+1}^{1-\gamma}\mu_{q,i}}\lesssim\lambda_{q,i+1}^{N+r+\gamma}\lambda_{q,i}^{\frac{1}{2}}\lambda_{q,i+1}^{-\frac{1}{2}}\delta_q^{\frac{1}{4}}\delta_{q+1}^{\frac{1}{2}}.
\end{align*} 
For the other term $R_{L2}$ defined in \eqref{def of R_L2}, we rewrite it as 
\begin{align*}
R_{L2}
	&=\sum_{I}\frac{1}{\lambda_{q,i+1}[I]}\tr(2\nabla u_{\ell,i} ((w_{A_I,f_{i+1}}+\overline{w}_{A_I,f_{i+1}})\otimes\nabla\gamma_{q,i+1,I})^{\top})\Id\\
	&\quad-\sum_{I}\frac{1}{\lambda_{q,i+1}[I]} ((w_{A_I,f_{i+1}}+\overline{w}_{A_I,f_{i+1}})\otimes\nabla\gamma_{q,i+1,I})^{\top}\nabla u_{\ell,i}+(\nabla u_{\ell,i})^{\top}( (w_{A_I,f_{i+1}}+\overline{w}_{A_I,f_{i+1}})\otimes\nabla\gamma_{q,i+1,I})\\
	&\quad+\sum_{I}\frac{\gamma_{q,i+1,I}}{\lambda_{q,i+1}[I]}\tr(2 (\nabla u_{\ell,i}-\nabla u_{\ell,i}^{s, \upsilon})(\nabla (w_{A_I,f_{i+1}}+\overline{w}_{A_I,f_{i+1}}))^{\top})\Id\\
	&\quad-\sum_{I}\frac{\gamma_{q,i+1,I}}{\lambda_{q,i+1}[I]}(\nabla (w_{A_I,f_{i+1}}+\overline{w}_{A_I,f_{i+1}}))^{\top}(\nabla u_{\ell,i}-\nabla u_{\ell,i}^{s, \upsilon})+( \nabla u_{\ell,i}- \nabla u_{\ell,i}^{s, \upsilon})^{\top}\nabla (w_{A_I,f_{i+1}}+\overline{w}_{A_I,f_{i+1}}).
\end{align*}
Notice that, for $(t,x)\in \supp\theta_I(t)\chi_I(x),$ $I=(s,\upsilon)\in\mathscr{I},$ we have
$
|t-s\tau_{q,i}|\lesssim\tau_{q,i}, \ |x-\upsilon \mu_{q,i}|\lesssim\mu_{q,i},
$
and then 
\begin{align*}
\|\partial_t^r(\nabla u_{\ell,i}(t,x)-\nabla u_{\ell,i}^{s, \upsilon})\|_N&\lesssim\|\partial_t^r(\nabla u_{\ell,i}(t,x)-\nabla u_{\ell,i}(s\tau_{q,i},x))\|_N+\|\partial_t^r(\nabla u_{\ell,i}(s\tau_{q,i},x)- \nabla u_{\ell,i}(s\tau_{q,i}, \upsilon\mu_{q,i}))\|_N\\
&\lesssim\|\partial_t^{r+1} u_{\ell,i}\|_{C^{0}(\cI^{q,i}_{3\ell_{q,i}};C^{N+1}(\T^3))}|t-s\tau_{q,i}|+\|\partial_t^{r} u_{\ell,i}\|_{C^{0}(\cI^{q,i}_{3\ell_{q,i}};C^{N+2}(\T^3))}|x-\upsilon\mu_{q,i}|\\
&\lesssim\ell_{q,i}^{-N-r}\lambda_{q,i}\delta_{q,i}^{\frac{1}{2}}(\tau_{q,i}+\mu_{q,i}),\quad N+r\geqslant0,
\end{align*}
where we have used \eqref{est on u_li 0}--\eqref{est on u_li 1}. 
So we could obtain
\begin{align*}
	\left\|\partial_t^rR_{L2}\right\|_{N}&\lesssim\sum_{N_1+N_2+N_3=N}\sum_{r_1+r_2+r_3=r}\lambda_{q,i+1}^{-1}\|\partial_t^{r_1}(\nabla u_{\ell,i}(t,x))\|_{N_1}\|\partial_t^{r_2}w_{A_I,f_{i+1}}\|_{N_2}\|\partial_t^{r_3}\nabla\gamma_{q,i+1,I}\|_{N_3}\\
	&\quad+\sum_{N_1+N_2+N_3=N}\sum_{r_1+r_2+r_3=r}\lambda_{q,i+1}^{-1}\|\partial_t^{r_1}(\nabla u_{\ell,i}(t,x)-\nabla u_{\ell,i}^{s, \upsilon})\|_{N_1}\|\partial_t^{r_2}(\nabla w_{A_I,f_{i+1}})\|_{N_2}\|\partial_t^{r_3}\gamma_{q,i+1,I}\|_{N_3}\\
	&\lesssim\lambda_{q,i+1}^{N+r}\left(\frac{\delta_{q+1}^{\frac{1}{2}}}{\lambda_{q,i+1}\mu_{q,i}}+(\tau_{q,i}+\mu_{q,i})\lambda_{q,i}\delta_{q,i}^{\frac{1}{2}}\delta_{q+1}^{\frac{1}{2}}\right)\\
	&\lesssim\lambda_{q,i+1}^{N+r}\lambda_{q,i}^{\frac{1}{2}}\lambda_{q,i+1}^{-\frac{1}{2}}\delta_q^{\frac{1}{4}}\delta_{q+1}^{\frac{1}{2}}.
\end{align*}
\subsection{Estimates on the mediation error}
Recall that 
\begin{align*}
R_{M}&=\tr(2\nabla (u_{q,i}-u_{\ell,i})(\nabla \tu_{q,i+1})^{\top})\Id-(\nabla \tu_{q,i+1})^\top\nabla (u_{q,i}-u_{\ell,i})-(\nabla (u_{q,i}-u_{\ell,i}))^\top\nabla \tu_{q,i+1}.
\end{align*}
By using \eqref{est on u_qi-u_li}, \eqref{est on u_qi-u_li}, and \eqref{est on tu_q}, we have 
\begin{align*}
\left\|\partial_t^rR_{M}\right\|_{N}&\lesssim\sum_{N_1+N_2=N}\sum_{r_1+r_2=r}\|\partial_t^{r_1}(u_{q,i}-u_{\ell,i})\|_{N_1+1}\|\partial_t^{r_2}\tu_{q,i+1}\|_{N_2+1}
\lesssim\lambda_{q,i+1}^{N+r}\ell_{q,i}^2\lambda_{q,i}^2\delta_{q,i}^{\frac{1}{2}}\delta_{q+1}^{\frac{1}{2}}\lesssim\lambda_{q,i+1}^{N+r}\lambda_{q,i}^{\frac{1}{2}}\lambda_{q,i+1}^{-\frac{1}{2}}\delta_q^{\frac{1}{4}}\delta_{q+1}^{\frac{1}{2}}.
\end{align*}
\subsection{Estimates on the oscillation error}
After simple calculation, we have
$$
|f_{i+1}|^2([I]\pm [J])f_{i+1}=(f_{i+1}\cdot([I]\pm [J])f_{i+1})f_{i+1},\quad \text{if}\ \|I-J\|\leqslant1.
$$ 
Combining it with $\supp u_{q,i+1,I}\bigcap\supp u_{q,i+1,J}=\emptyset,$ for $\|I-J\|>1,$ we could obtain
\begin{align*}
	&\quad\Div\left(\tr(\tilde{w}_{q,i+1,p}(\tilde{w}_{q,i+1,p})^{\top})\Id-(\tilde{w}_{q,i+1,p})^\top\tilde{w}_{q,i+1,p}-2|f_{i+1}|^2\sum_I|u_{q,i+1,I}|^2\left(\Id-\frac{f_{i+1}}{|f_{i+1}|}\otimes \frac{f_{i+1}}{|f_{i+1}|}\right)\right)\\
	&=\Div\left(\sum_I\sum_{J}|f_{i+1}|^2u_{q,i+1,I}\cdot u_{q,i+1,J}\left(e^{i\lambda_{q,i+1}(\xi_{A_I,f_{i+1}}-\xi_{A_J,f_{i+1}})}+e^{i\lambda_{q,i+1}(-\xi_{A_I,f_{i+1}}+\xi_{A_J,f_{i+1}})}\right)\Id\right)\\
	&\quad-\Div\left(\sum_I\sum_{J}|f_{i+1}|^2u_{q,i+1,I}\cdot u_{q,i+1,J}\left(e^{i\lambda_{q,i+1}(\xi_{A_I,f_{i+1}}+\xi_{A_J,f_{i+1}})}+e^{-i\lambda_{q,i+1}(\xi_{A_I,f_{i+1}}+\xi_{A_J,f_{i+1}})}\right)\Id\right)\\
	&\quad-\Div\left(\sum_I\sum_{J}u_{q,i+1,I}\cdot u_{q,i+1,J}\left(e^{i\lambda_{q,i+1}(\xi_{A_I,f_{i+1}}-\xi_{A_J,f_{i+1}})}+e^{i\lambda_{q,i+1}(-\xi_{A_I,f_{i+1}}+\xi_{A_J,f_{i+1}})}\right)f_{i+1}\otimes f_{i+1}\right)\\
	&\quad+\Div\left(\sum_I\sum_{J}u_{q,i+1,I}\cdot u_{q,i+1,J}\left(e^{i\lambda_{q,i+1}(\xi_{A_I,f_{i+1}}+\xi_{A_J,f_{i+1}})}+e^{-i\lambda_{q,i+1}(\xi_{A_I,f_{i+1}}+\xi_{A_J,f_{i+1}})}\right)f_{i+1}\otimes f_{i+1}\right)\\
	&\quad-\Div\left(\sum_{I}2|f_{i+1}|^2|u_{q,i+1,I}|^2\left(\Id-\frac{f_{i+1}}{|f_{i+1}|}\otimes\frac{f_{i+1}}{|f_{i+1}|}\right)\right)\\
	&=\sum_{\|I-J\|=1}U_{I,J}\left(e^{i\lambda_{q,i+1}(\xi_{A_I,f_{i+1}}-\xi_{A_J,f_{i+1}})}+e^{i\lambda_{q,i+1}(-\xi_{A_I,f_{i+1}}+\xi_{A_J,f_{i+1}})}-e^{i\lambda_{q,i+1}(\xi_{A_I,f_{i+1}}+\xi_{A_J,f_{i+1}})}-e^{-i\lambda_{q,i+1}(\xi_{A_I,f_{i+1}}+\xi_{A_J,f_{i+1}})}\right)\\
	&\quad-\sum_{I=J}\Div\left(\sum_{I}|f_{i+1}|^2|u_{q,i+1,I}|^2\left(\Id-\frac{f_{i+1}}{|f_{i+1}|}\otimes\frac{f_{i+1}}{|f_{i+1}|}\right)(e^{i2\lambda_{q,i+1}\xi_{A_I,f_{i+1}}}+e^{-i2\lambda_{q,i+1}\xi_{A_I,f_{i+1}}})\right)\\
	&=\sum_{\|I-J\|=1}U_{I,J}\left(e^{i\lambda_{q,i+1}(\xi_{A_I,f_{i+1}}-\xi_{A_J,f_{i+1}})}+e^{i\lambda_{q,i+1}(-\xi_{A_I,f_{i+1}}+\xi_{A_J,f_{i+1}})}-e^{i\lambda_{q,i+1}(\xi_{A_I,f_{i+1}}+\xi_{A_J,f_{i+1}})}-e^{-i\lambda_{q,i+1}(\xi_{A_I,f_{i+1}}+\xi_{A_J,f_{i+1}})}\right)\\
	&\quad-\sum_{I}U_{I,I}\left(e^{i2\lambda_{q,i+1}\xi_{A_I,f_{i+1}}}+e^{-i2\lambda_{q,i+1}\xi_{A_I,f_{i+1}}}\right),
\end{align*}
where
$
U_{I,J}=|f_{i+1}|^2\nabla(u_{q,i+1,I}\cdot u_{q,i+1,J})-(f_{i+1}\cdot\nabla)(u_{q,i+1,I}\cdot u_{q,i+1,J})f_{i+1}
.$ Then, we could use  $\nabla\xi_{A_I,f_{i+1}}=[I]f_{i+1},$ and Lemma \ref{est on int operator} to obtain
\begin{align*}
	\left\|\partial_t^rR_{O1}\right\|_{N+\gamma}
	&\lesssim\lambda_{q,i+1}^{N+r}\sum_{\|I-J\|\leqslant1}\sum_{r_1+r_2=r}\frac{\|u_{q,i+1,I}\|_{1}\|u_{q,i+1,J}\|_{0}+\|u_{q,i+1,J}\|_{1}\|u_{q,i+1,I}\|_{0}}{\lambda_{q,i+1}^{1-\gamma}}\\
	&\quad+\lambda_{q,i+1}^{N+r}\sum_{\|I-J\|\leqslant1}\sum_{r_1+r_2=r}\sum_{N_1+N_2=n_0+1}\frac{\|u_{q,i+1,I}\|_{N_1+1+\gamma}\|u_{q,i+1,J}\|_{N_2+\gamma}+\|u_{q,i+1,J}\|_{N_1+1+\gamma}\|u_{q,i+1,I}\|_{N_2+\gamma}}{\lambda_{q,i+1}^{n_0-\gamma}}\\
	&\lesssim\lambda_{q,i+1}^{N+r}\left(\frac{\delta_{q+1}}{\lambda_{q,i+1}^{1-\gamma}\mu_{q,i}}+\frac{\delta_{q+1}}{\lambda_{q,i+1}^{n_0-\gamma}\mu_{q,i}^{n_0+1+2\gamma}}\right)\lesssim\lambda_{q,i+1}^{N+r}\frac{\delta_{q+1}}{\lambda_{q,i+1}^{1-\gamma}\mu_{q,i}}\lesssim\lambda_{q,i+1}^{N+r+\gamma}\lambda_{q,i}^{\frac{1}{2}}\lambda_{q,i+1}^{-\frac{1}{2}}\delta_q^{\frac{1}{4}}\delta_{q+1}.
\end{align*}
Combine it with 
\begin{align*}
\|\partial_t^rR_{O2}\|_{N}&\lesssim\sum_{N_1+N_2=N}\sum_{r_1+r_2=r}\left(\|\nabla \partial^{r_1}_t\tilde{w}_{q,i+1,p}\|_{N_1}\|\nabla \partial^{r_2}_t\tilde{w}_{q,i+1,c}\|_{N_2}+\|\nabla \partial^{r_1}_t\tilde{w}_{q,i+1,c}\|_{N_1}\|\nabla \partial^{r_2}_t\tilde{w}_{q,i+1,c}\|_{N_2}\right)\\
&\lesssim\lambda_{q,i+1}^{N+r}\left(\frac{\delta_{q+1}}{\lambda_{q,i+1}\mu_{q,i}}+\frac{\delta_{q+1}}{(\lambda_{q,i+1}\mu_{q,i})^2}\right)\lesssim\lambda_{q,i+1}^{N+r+\gamma}\lambda_{q,i}^{\frac{1}{2}}\lambda_{q,i+1}^{-\frac{1}{2}}\delta_q^{\frac{1}{4}}\delta_{q+1},
\end{align*}
where we have used \eqref{est on nabla tu_q p} and \eqref{est on nabla tu_q c}, we could get
\begin{align*}
		\|\partial_t^rR_{O}\|_{N}\leqslant\|\partial_t^rR_{O1}\|_{N}+\|\partial_t^rR_{O2}\|_{N} \lesssim\lambda_{q,i+1}^{N+r+\gamma}\lambda_{q,i}^{\frac{1}{2}}\lambda_{q,i+1}^{-\frac{1}{2}}\delta_q^{\frac{1}{4}}\delta_{q+1}.
\end{align*}
To sum up, we have \eqref{est on delta R_q i+1}. \eqref{est on R_q-R_l} can be deduced from \eqref{est on R_qi-R_li}.
\section{Proof of the main theorem}\label{Proof of the main theorem}
\subsection{Proof of Proposition \ref{Inductive proposition}}
For any $0 < \beta < \frac{1}{60},$ let the parameters $\varepsilon_1,$ $\bar{b}(\beta) = \frac{1 + 12 \beta - 36 \gamma}{48 \beta},$ and $a^*_0$ be as in the statement of Proposition \ref{prop of Reynolds error}. For any $a > a^*_0$ and $\varepsilon < \varepsilon_1,$ given an approximate solution $(u_q, c_q, R_q)$ defined on $\cI^{q,-1} \times \T^3,$ we have constructed a perturbation $\tilde{u}_q = \sum_{i=1}^6 \tilde{u}_{q,i}$ and will apply it to $u_q$. This creates a new Reynolds stress $R_{q+1},$ which satisfies the estimates in Proposition \ref{prop of Reynolds error}. We now need to confirm whether $(u_{q+1}, c_{q+1}, R_{q+1})$ satisfy \eqref{est on u_q}--\eqref{est on R_q} at the $q+1$ step.

First, we denote the implicit constants in \eqref{est on tu_q}, which does not depend on $M,$ by $M_0$ and set $M = 120M_0$. If we set $\|\cdot\|_N = \|\cdot\|_{C^0(\cI^{q,-1}; C^N(\T^3))},$ we have

\begin{equation}\label{Induction Proposition}
	\begin{aligned}
		\sum_{N+r\leqslant3}\lambda_{q+1}^{1-N-r}\|\partial_{t}^r(u_{q+1}-u_q)\|_N\leqslant\sum_{N+r\leqslant3}\sum_{i=1}^6\lambda_{q,i}^{1-N-r}\|\partial_{t}^r\tu_{q,i}\|_N\leqslant60M_0\delta_{q+1}^{\frac{1}{2}}\leqslant M\delta_{q+1}^{\frac{1}{2}}.
	\end{aligned}
\end{equation}
Moreover, we could get for sufficient large $a$ and for $2\leqslant N+r\leqslant 3$,
\begin{align*}
	\|u_{q+1}\|_0&\leqslant\|u_{q}\|_0+\sum_{i=1}^6\|\tu_{q,i}\|_0\leqslant \varepsilon-\delta_{q}^{\frac{1}{2}}+M_0\sum_{i=1}^6\lambda_{q,i}^{-1}\delta_{q+1}^{\frac{1}{2}}\leqslant \varepsilon-\delta_{q+1}^{\frac{1}{2}}, \\
	\|\nabla u_{q+1}\|_0&\leqslant\|\nabla u_{q}\|_0+\sum_{i=1}^6\|\nabla  \tu_{q,i}\|_0\leqslant \varepsilon-\delta_{q}^{\frac{1}{2}}+M_0\sum_{i=1}^6\delta_{q+1}^{\frac{1}{2}}\leqslant \varepsilon-\delta_{q+1}^{\frac{1}{2}}, \\
	\|\partial_t u_{q+1}\|_0&\leqslant\|\partial_t u_{q}\|_0+\sum_{i=1}^6\|\partial_t \tu_{q,i}\|_0\leqslant \varepsilon-\delta_{q}^{\frac{1}{2}}+M_0\sum_{i=1}^6\delta_{q+1}^{\frac{1}{2}}\leqslant \varepsilon-\delta_{q+1}^{\frac{1}{2}}, \\
	\|\partial_{t}^ru_{q+1}\|_N&\leqslant\|\partial_{t}^ru_{q}\|_N+\sum_{i=1}^6\|\partial_{t}^r\tu_{q,i}\|_N\leqslant M\lambda_{q}^{N+r-1}\delta_{q}^{\frac{1}{2}}+\sum_{i=1}^6\frac{1}{120}M\lambda_{q,i}^{N+r-1}\delta_{q+1}^{\frac{1}{2}}\leqslant M\lambda_{q+1}^{N+r-1}\delta_{q+1}^{\frac{1}{2}}.
\end{align*}
The estimates on $R_{q+1}$ can be obtained from Proposition \ref{prop of Reynolds error}. Let the new Reynolds error be
\begin{align*}
	R_{q+1}&=R_{q.6}-\delta_{q+1}\Id\\
	&=R_{q}-\delta_{q+1}\Id+\sum_{i=1}^6(\delta_{q+1}^{\frac{1}{2}}\Gamma_{f_{i}}(\Id-\delta_{q+1}^{-1}R_{\ell}))^2\left(\Id-\frac{f_{i}}{|f_{i}|}\otimes \frac{f_{i}}{|f_{i}|}\right)+\sum_{i=1}^6\delta R_{q,i}\\
	&=R_{q}-R_{\ell}+\sum_{i=1}^6\delta R_{q,i}.
\end{align*}
Then, we could use \eqref{est on delta R_q i+1} and \eqref{est on R_q-R_l} to obtain for $0\leqslant N+r\leqslant 2$
\begin{align*}
\|\partial_t^rR_{q+1}\|_{N}&\leqslant\|\partial_t^r(R_{q}-R_{\ell})\|_{N}+\sum_{i=1}^6\|\partial_t^r\delta R_{q,i}\|_{N}\leqslant \lambda_{q+1}^{N+r-2\gamma} \delta_{q+2}.
\end{align*}
\subsection{Proof of Proposition \ref{Bifurcating inductive proposition}}
Similar to \cite{DK22,GK22}, we consider a given time interval $\cal I \subset (0,T)$ with $|\cal I| \geqslant 3\tau_{q,-1},$ within which we can always find $s_0$ such that $\supp(\theta_{s_0}(\tau_{q,5}^{-1}\cdot)) \subset \cal I$. If $I = (s_0, \upsilon) \in \mathscr{I},$ we replace $\Gamma_{f_6}$ in $\tilde{u}_{q,6}$ with $\tilde{\Gamma}_{f_6} = -\Gamma_{f_6},$ which will make $\tilde{\gamma}_{q,6,I} = -\gamma_{q,6,I}$. We denote the new perturbation by $\tilde{u}_{q,6,\text{new}}$. As for the other tuples, we do not change $\gamma_{q,6,I}$. Note that $\tilde{\Gamma}_{f_6}^2 = \Gamma_{f_6}^2,$ this allows us to use the same formula \eqref{Geometric lemma 1}, and the replacement does not change the estimates on $\tilde{\Gamma}_{f_6}$. Therefore, the estimates on the new perturbation $\tilde{u}_{q,6,\text{new}}$ are the same as those on $\tilde{u}_{q,6}$. Up to now, we could construct the new corrected approximate solution $(\overline{u}_{q+1}, c_{q+1}, \overline{R}_{q+1})$ that satisfies \eqref{est on u_q}--\eqref{est on R_q} at the $q+1$ step, where $\overline{u}_{q+1} = u_q + \sum_{i=1}^{5}\tilde{u}_{q,i} + \tilde{u}_{q,6,\text{new}}$. By the construction, the correction $\tilde{u}_{q,6,\text{new}}$ differs from $\tilde{u}_{q,6}$ on the support of $\theta_{s_0}(\tau_{q,5}^{-1}\cdot)$. Therefore, $\supp_t(\overline{u}_{q+1} - u_{q+1}) = \supp_t(\tilde{u}_{q,6} - \tilde{u}_{q,6,\text{new}}) \subset \cal I$. 
Recall the definition of $\tu_{q,6}$ ,
\begin{align*}
\tilde{u}_{q,6}&=\Large\sum_{I\in \mathscr I: s_I = s_0}\left(
\frac{\theta_I(t)\chi_I(x)d_{q,6}\tilde{f}_{A_I}}{\sqrt{2}\lambda_{q+1}[I]|\tilde{f}_{A_I}||f_{6}|}(e^{i\lambda_{q,i+1}\xi_{A_I,f_{6}}}+e^{-i\lambda_{q,i+1}\xi_{A_I,f_{6}}})-\int_{\T^3}\frac{\theta_I(t)\chi_I(x)d_{q,6}\tilde{f}_{A_I}}{\sqrt{2}\lambda_{q+1}[I]|\tilde{f}_{A_I}||f_{6}|}(e^{i\lambda_{q,i+1}\xi_{A_I,f_{6}}}+e^{-i\lambda_{q,i+1}\xi_{A_I,f_{6}}})\rd x\right),
\end{align*}
and $\tu_{q,6,new}=-\tu_{q,6}$, we could obtain
\begin{align*}
	&\quad|\tu_{q,6,new} -  \tu_{q,6} |^2\\
	&=  \sum_{I\in \mathscr I: s_I = s_0}\left(
	\frac{2\theta_I^2(t) \chi_I^2(x) d_{q,6}^2}{\lambda_{q+1}^2[I]^2|f_6|^2}  (e^{i2\lambda_{q+1}\xi_{A_I,f_{6}}}+e^{-i2\lambda_{q+1}\xi_{A_I,f_{6}}}+2)+2\left(\int_{\T^3}\frac{\theta_I(t)\chi_I(x)d_{q,6}\tilde{f}_{A_I}}{\lambda_{q+1}[I]|\tilde{f}_{A_I}||f_{6}|}(e^{i\lambda_{q,i+1}\xi_{A_I,f_{6}}}+e^{-i\lambda_{q,i+1}\xi_{A_I,f_{6}}})\rd x\right)^2\right)\\
	&\quad-\sum_{I\in \mathscr I: s_I = s_0}\frac{4\theta_I(t)\chi_I(x)d_{q,6}\tilde{f}_{A_I}}{\lambda_{q+1}[I]|\tilde{f}_{A_I}||f_{6}|}(e^{i\lambda_{q,i+1}\xi_{A_I,f_{6}}}+e^{-i\lambda_{q,i+1}\xi_{A_I,f_{6}}})\left(\int_{\T^3}\frac{\theta_I(t)\chi_I(x)d_{q,6}\tilde{f}_{A_I}}{\lambda_{q+1}[I]|\tilde{f}_{A_I}||f_{6}|}(e^{i\lambda_{q,i+1}\xi_{A_I,f_{6}}}+e^{-i\lambda_{q,i+1}\xi_{A_I,f_{6}}})\rd x\right).
\end{align*}
Notice that 
\begin{align*}
d_{q,6}^2&=\delta_{q+1}\Gamma_6^2(\Id-\delta_{q+1}^{-1}R_\ell)\\
&=\frac{\delta_{q+1}}{4}(3(\Id-\delta_{q+1}^{-1}R_\ell)_{33}+4(\Id-\delta_{q+1}^{-1}R_\ell)_{12}-(\Id-\delta_{q+1}^{-1}R_\ell)_{11}-(\Id-\delta_{q+1}^{-1}R_\ell)_{22})\\
&=\frac{\delta_{q+1}}{4}(1-3(\delta_{q+1}^{-1}R_\ell)_{33}-4(\delta_{q+1}^{-1}R_\ell)_{12}+(\delta_{q+1}^{-1}R_\ell)_{11}+(\delta_{q+1}^{-1}R_\ell)_{22})\\
&\geqslant\frac{\delta_{q+1}}{4}(1-9r_0)\geqslant\frac{\delta_{q+1}}{8},
\end{align*}
we could use Lemma \ref{est on int operator} to get
\begin{align*}
	&\quad\sum_{I\in \mathscr I: s_I = s_0}\int_{\T^3}
	\frac{2\theta_I^2(t) \chi_I^2(x) d_{q,6}^2}{\lambda_{q+1}^2[I]^2|f_6|^2}  (e^{i2\lambda_{q+1}\xi_{A_I,f_{6}}}+e^{-i2\lambda_{q+1}\xi_{A_I,f_{6}}})\rd x\\
	&\lesssim_{n_0}\sum_{I\in \mathscr I: s_I = s_0}\frac{2}{\lambda_{q+1}^2[I]^2|f_6|^2}\cdot \frac{\|\theta_I^2(t) \chi_I^2(x) d_{q,6}^2\|_0+\|\theta_I^2(t) \chi_I^2(x) d_{q,6}^2\|_{n_0}}{(\lambda_{q+1}|f_6|)^{n_0}}\\
	&\lesssim_{n_0}\frac{\delta_{q+1}}{\lambda_{q+1}^{n_0+2}\mu_{q,5}^{n_0}}	\lesssim_{n_0}\frac{\delta_{q+1}}{\lambda_{q+1}^{3}}\leqslant\frac{\pi^3 \delta_{q+1}}{256\lambda_{q+1}^2|f_6|^2},
\end{align*}
and similarly,
\begin{align*}
	&\quad\sum_{I\in \mathscr I: s_I = s_0}\int_{\T^3}
	\frac{\theta_I(t)\chi_I(x)d_{q,6}\tilde{f}_{A_I}}{\lambda_{q+1}[I]|\tilde{f}_{A_I}||f_{6}|}(e^{i\lambda_{q,i+1}\xi_{A_I,f_{6}}}+e^{-i\lambda_{q,i+1}\xi_{A_I,f_{6}}})\rd x\\
	&\lesssim_{n_0}\sum_{I\in \mathscr I: s_I = s_0}\frac{1}{\lambda_{q+1}[I]|f_6|}\cdot \frac{\|\theta_I(t) \chi_I(x) d_{q,6}\|_0+\|\theta_I(t) \chi_I(x) d_{q,6}\|_{n_0}}{(\lambda_{q+1}|f_6|)^{n_0}}\\
	&\lesssim_{n_0}\frac{\delta_{q+1}^{\frac{1}{2}}}{\lambda_{q+1}^{n_0+1}\mu_{q,5}^{n_0}}	\lesssim_{n_0}\frac{\delta_{q+1}^{\frac{1}{2}}}{\lambda_{q+1}^{2}}\leqslant\frac{\pi^{\frac{3}{2}} \delta_{q+1}^{\frac{1}{2}}}{16\sqrt{2}\lambda_{q+1}|f_6|},
\end{align*}
for sufficiently large $a$. Then, for $t \in \supp \theta_{s_0}(\tau_{q,5}^{-1}\cdot) $ and satisfy $\theta_I(t)= \theta_{s_0}(\tau_{q,5}^{-1}t)=1,$ we have,
\begin{align*}
	\int_{\T^3}|\tu_{q,6,new} -  \tu_{q,6} |^2dx&=\sum_{I\in \mathscr I: s_I = s_0}\int_{\T^3}\frac{4\theta_I^2(t) \chi_I^2(x) d_{q,6}^2}{\lambda_{q+1}^2[I]^2|f_6|^2}\rd x+\sum_{I\in \mathscr I: s_I = s_0}\int_{\T^3}
	\frac{2\theta_I^2(t) \chi_I^2(x) d_{q,6}^2}{\lambda_{q+1}^2[I]^2|f_6|^2}  (e^{i2\lambda_{q+1}\xi_{A_I,f_{6}}}+e^{-i2\lambda_{q+1}\xi_{A_I,f_{6}}})\rd x
	\\
	&\quad-2\left(\int_{\T^3}\frac{\theta_I(t)\chi_I(x)d_{q,6}\tilde{f}_{A_I}}{\lambda_{q+1}[I]|\tilde{f}_{A_I}||f_{6}|}(e^{i\lambda_{q,i+1}\xi_{A_I,f_{6}}}+e^{-i\lambda_{q,i+1}\xi_{A_I,f_{6}}})\rd x\right)^2\\
	&\geqslant\frac{\pi^3 \delta_{q+1}}{64\lambda_{q+1}^2|f_6|^2}-\frac{\pi^3 \delta_{q+1}}{128\lambda_{q+1}^2|f_6|^2}=\frac{\pi^3 \delta_{q+1}}{256\lambda_{q+1}^2},
\end{align*}
where we have used $[I]\leqslant 16$. Therefore, we obtain
\begin{align*}
	\|\overline{u}_{q+1} -  u_{q+1} \|_{C^0([0,T]; L^2(\T^3))}
	=\|\tu_{q,6,new}-\tu_{q,6}\|_{C^0([0,T]; L^2(\T^3))}\geqslant \frac{\pi^{\frac{3}{2}} \delta_{q+1}^{\frac{1}{2}}}{16\lambda_{q+1}}.
\end{align*}
Assume we are given an  approximate solutions $(u_q, c_q, R_q)$  that satisfies \eqref{est on u_q}--\eqref{est on R_q} and
$$
\supp_t(\overline{u}_q - u_q, \overline{R}_q - R_q) \subset \mathcal{J},
$$
for some time interval $\mathcal{J}$. We can then construct the regularized Reynolds errors $R_\ell$ and $\overline{R}_\ell$. Notice that they differ only in $\mathcal{J} + \ell_{q,0} \subset \mathcal{J} + (\lambda_q \delta_q^{\frac{1}{4}})^{-1}$. As a result, $\tilde{u}_{q,6,\text{new}}$ differs from $\tilde{u}_{q,6}$ in $\mathcal{J} + (\lambda_q \delta_q^{\frac{1}{4}})^{-1},$ and we can obtain different corrected approximate solutions $(\overline{u}_{q+1}, c_{q+1}, \overline{R}_{q+1})$ and $(u_{q+1}, c_{q+1}, R_{q+1})$ satisfying 
$$
\supp_t(\overline{u}_{q+1} - u_{q+1}, \overline{R}_{q+1} - R_{q+1}) \subset \mathcal{J} + (\lambda_q \delta_q^{\frac{1}{4}})^{-1}.
$$
\subsection{Proof of Theorem \ref{thm 1} }\label{Proof of the theorems}
For convenience, we assume that $T \geqslant 20$ in this argument. We fix $0 < \alpha < \frac{1}{60}$ and $\beta \in (\alpha, \frac{1}{60})$. We choose $\varepsilon,$ $b,$ and $a$ based on Proposition \ref{Inductive proposition}. In Section \ref{Construction of the starting tuple}, we constructed an initial approximate solution $(u_0, c_0, R_0)$ which solves \eqref{approximation elstro dynamic} on $\cI^{-1} \times \T^3$ and satisfies \eqref{est on u_q} and \eqref{est on R_q}.

We can apply Proposition \ref{Inductive proposition} iteratively to produce a sequence of approximate solutions $(u_q, c_q, R_q),$ which solve \eqref{approximation elstro dynamic} and satisfy \eqref{est on u_q}--\eqref{Proposition of induction}.

First, we prove that $u_q$ is a Cauchy sequence in $C^0([0, T]; C^{1, \alpha}(\T^3))$. For any $q \leqslant q',$ we have
\begin{align*}
	\|u_{q'} - u_q\|_{C^0([0, T]; C^{1, \alpha}(\T^3))}
	&\leqslant \sum_{l=1}^{q'-q} \|u_{q+l} - u_{q+l-1}\|_{C^0([0, T]; C^{1, \alpha}(\T^3))} 
	\leqslant \sum_{l=1}^{q'-q} \|u_{q+l} - u_{q+l-1}\|_1^{1-\alpha} \|u_{q+l} - u_{q+l-1}\|_2^\alpha 
	= \sum_{l=1}^{q'-q} \lambda_{q+l}^{\alpha-\beta}.
\end{align*}
Notice $\sum_{l=1}^{q'-q} \lambda_{q+l}^{\alpha-\beta}$ will converge to $0$ as $q$ goes to infinity. Thus, $u_q$ converges to a limit 
$
u \in C^0([0, T], C^{1, \alpha}(\T^3)).
$
Similarly, the time regularity follows from \eqref{Proposition of induction} such that
$
u \in C^{1, \alpha}([0, T], C^0(\T^3)).
$
Hence,
$
u \in C^{1, \alpha'}([0, T] \times \T^3),
$
for $\alpha' < \alpha < \beta < \frac{1}{60}$. Moreover, $c_q$ and $R_q$ converge to $0$ in $C^0([0, T] \times \T^3)$.

Next, we construct two distinct tuples by using Proposition \ref{Bifurcating inductive proposition}. At the $\bar{q}$-th step, we produce two distinct tuples $(u_q, c_q, R_q)$ and $(\overline{u}_q, c_q, \overline{R}_q)$ which satisfy Proposition \ref{Bifurcating inductive proposition}, and we have 
\begin{align*}
	&\|\overline{u}_{\bar{q}} - u_{\bar{q}}\|_{C^0([0, T]; L^2(\T^3))} \geqslant \frac{\pi^{\frac{3}{2}} \delta_{q+1}^{\frac{1}{2}}}{16 \lambda_{q+1}}, \quad \supp_t (u_{\bar{q}} - \overline{u}_{\bar{q}}) \subset \mathcal{I},
\end{align*}
with $\mathcal{I} = (10, 10 + 3 \tau_{\bar{q}, -1})$.

Next, we apply Proposition \ref{Inductive proposition} iteratively to build a new sequence $(\overline{u}_q, c_q, \overline{R}_q)$ which satisfies \eqref{est on u_q}, \eqref{est on R_q}, and \eqref{Proposition of induction}. This new sequence also converges to a solution $\overline{u}$ to the quasi-linear wave equations, and
$
\overline{u} \in C^{1, \alpha'}([0, T] \times \T^3).
$
Moreover, $\overline{u}_q$ shares initial data with $u_q$ for all $q,$ because for any $q \geqslant \bar{q}$ and for $a$ large enough, we have
$$
\supp_t(u_q - \overline{u}_q) \subset \mathcal{I} + \sum_{q=\bar{q}}^\infty (\lambda_q \delta_q^\frac{1}{4})^{-1} \subset [9, T],
$$
and thus the two solutions $\overline{u}$ and $u$ have the same initial data. However, $\overline{u}$ differs from $u$ because
\begin{align*}
	\|u - \overline{u}\|_{C^0([0, T]; L^2(\T^3))}
	&\geqslant \|u_{\bar{q}} - \overline{u}_{\bar{q}}\|_{C^0([0, T]; L^2(\T^3))} - \sum_{q=\bar{q}}^\infty \|u_{q+1} - u_q - (\overline{u}_{q+1} - \overline{u}_q)\|_{C^0([0, T]; L^2(\T^3))} \\
	&\geqslant \|u_{\bar{q}} - \overline{u}_{\bar{q}}\|_{C^0([0, T]; L^2(\T^3))} - (2\pi)^{\frac{3}{2}} \sum_{q=\bar{q}}^\infty (\|u_{q+1} - u_q\|_0 + \|\overline{u}_{q+1} - \overline{u}_q\|_0) \\
	&\geqslant \frac{\pi^{\frac{3}{2}} \delta_{\bar{q}}^{\frac{1}{2}}}{16 \lambda_{\bar{q}}} - 2 (2\pi)^{\frac{3}{2}} M \sum_{q=\bar{q}}^\infty \frac{\delta_{q+1}^\frac{1}{2}}{\lambda_{q+1}} > 0,
\end{align*}
if we choose $a$ large enough.

By changing the choice of time interval $\mathcal{I}$ and the choice of $\bar{q},$ we can generate infinitely many solutions in a similar way.

\appendix
\section{H\"{o}lder spaces}
In this section, we introduce the notations we would use for H{\"o}lder spaces. For some time interval $\cI\subset\R,$ we denote the supremum norm as  $\|f\|_0=\left\|f\right\|_{C^0(\cI;C^0(\T^3))}:=\underset{(t,x)\in\cI\times\T^3}{\sup}|f(t,x)|$. For $N\in\N,$ a multi-index $k=(k_1,k_2,k_3)\in\N^3$ and $\alpha\in(0,1],$ we denote the H\"{o}lder seminorms as
$$
\begin{aligned}
	&[f]_{N}=\underset{|k|=N}{\max}\left\|D^k f\right\|_{0},&&
	[f]_{N+\alpha}=\underset{|k|=N}{\max}\underset{t}{\sup}\underset{x\neq y}{\sup}\frac{|D^k f(t,x)-D^k f(t,x)|}{|x-y|^\alpha},
\end{aligned}
$$
where $D^k$ are spatial derivatives. Then we can denote the H\"{o}lder norms as
$$
\left\|f\right\|_{N}=\sum_{j=0}^{m}[f]_{j}, \qquad\left\|f\right\|_{N+\alpha}=\left\|f\right\|_{N}+[f]_{N+\alpha}.
$$
\section{Conservation law form of nonlinear equations}\label{Conservation law form of nonlinear equation}
Obviously, when the derivative of $u$ is small enough, the equations \eqref{system} can be transformed into quasilinear hyperbolic systems. In this section, we study the structure of the quasilinear wave equations \eqref{system} in 2D and 3D. We denote
$
\Omega_\varepsilon = \{u \in C^1 \mid \|\nabla u\|_0 + \|\partial_t u\|_0 \leqslant \varepsilon\},
$
and we will discuss the solution $u$ whose derivatives $\|\nabla u\|_0$ and $\|\partial_t u\|_0$ are small enough, i.e., $u \in \Omega_\varepsilon$ for sufficiently small $\varepsilon$.

We first calculate the 2-dimensional case. Let
$
U = (\partial_1 u_1, \partial_2 u_1, \partial_1 u_2, \partial_2 u_2, \partial_t u_1, \partial_t u_2)^{\top},
$
then the equations \eqref{system} can be written as:

\begin{align}
	\frac{\partial U}{\partial t}+\sum_{j=1}^2A_j(U)\frac{\partial U}{\partial x_j}=0, \label{system 2D}
\end{align}
where
\begin{align*}
	A_j(U)&=\left(\begin{matrix}
		0&E_j\\
		B_j&0\\
	\end{matrix}\right), \quad 1\leqslant i\leqslant 2,
\end{align*}
and
\begin{align*}
	E_1=\left(\begin{matrix}
	-1&0\\
	0&0\\
	0&-1\\
	0&0
\end{matrix}\right), \quad
	E_2=\left(\begin{matrix}
	0&0\\
	-1&0\\
	0&0\\
	0&-1
\end{matrix}\right),
\end{align*}
\begin{align*}
B_1=\left(\begin{matrix}
-\lambda-2\mu&2\partial_2 u_1&0&-\lambda-\mu+2\partial_2u_2\\
-\partial_2 u_1&-\partial_1u_1&-\mu-\partial_2 u_2&-\partial_1u_2
\end{matrix}\right), \quad
B_2=\left(\begin{matrix}
-\partial_2u_1&-\mu-\partial_1 u_1&-\partial_2u_2&-\partial_1 u_2\\
-\lambda-\mu+2\partial_1u_1&0&2\partial_1 u_2&-\lambda-2\mu
\end{matrix}\right).
\end{align*}
For $\xi=(\xi_1,\xi_2)\in \mathbb{S}^{1},$ denote
\begin{align*}
	A(U, \xi)=\sum_{j=1}^2A_{j}(U)\xi_j
	=\left(\begin{matrix}
		0&\sum_{j=1}^2\xi_jE_j\\
		\sum_{j=1}^2\xi_jB_j&0
	\end{matrix}\right).
\end{align*}
Next, we calculate the corresponding eigenvalues $\{\tilde{\lambda}_i\}_{i=1}^6$ and eigenvectors $\{r_i\}_{i=1}^6$. It is easy to get
\begin{align*}
	\tilde{\lambda}_i\Id-A(U, \xi)
	&:=\left(\begin{matrix}
		\tilde{\lambda}_i\Id_4&-\sum_{j=1}^2\xi_jE_j\\
		-\sum_{j=1}^2\xi_jB_j&\tilde{\lambda}_i\Id_2
	\end{matrix}\right),
\end{align*}
and then
\begin{align*}
	|\tilde{\lambda}_i\Id_6-A(U, \xi)|
	&=\tilde{\lambda}_i^2\left|\tilde{\lambda}_i^2\Id_2-C\right|\\
	&=\tilde{\lambda}_i^2\left(\tilde{\lambda}_i^2-\frac{c_{11}+c_{22}}{2}+\frac{1}{2}\sqrt{(c_{11}-c_{22})^2+4c_{12}c_{21}}\right)\left(\tilde{\lambda}_i^2-\frac{c_{11}+c_{22}}{2}-\frac{1}{2}\sqrt{(c_{11}-c_{22})^2+4c_{12}c_{21}}\right),
\end{align*}
where 
\begin{align*}
	C:=\sum_{i,j=1}^2\xi_i\xi_jB_iE_j&=\left(c_{ij}\right)_{i,j=1}^2=\left(\begin{matrix}
		(\lambda+2\mu)\xi_1^2-\partial_2 u_1\xi_1\xi_2+(\mu+\partial_1 u_1)\xi_2^2& (\lambda+\mu-\partial_2u_2)\xi_1\xi_2+\partial_1 u_2\xi_2^2\\
		\partial_2u_1\xi_1^2+(\lambda+\mu-\partial_1 u_1)\xi_1\xi_2&(\mu+\partial_2u_2)\xi_1^2-\partial_1 u_2\xi_1\xi_2+(\lambda+2\mu)\xi_2^2
	\end{matrix}\right).
\end{align*}
So 6 real eigenvalues of $A(U, \xi)$ can be represented as 
\begin{align*}
	\tilde{\lambda}_{1,2}=0, \quad\tilde{\lambda}_{3,4}=\pm\sqrt{\frac{c_{11}+c_{22}-\sqrt{(c_{11}-c_{22})^2+4c_{12}c_{21}}}{2}}, \quad\tilde{\lambda}_{5,6}=\pm\sqrt{\frac{c_{11}+c_{22}+\sqrt{(c_{11}-c_{22})^2+4c_{12}c_{21}}}{2}}.
\end{align*}
where $\tilde{\lambda}_{3,4}$ and $\tilde{\lambda}_{5,6}$ are the square root of the eigenvalue of $C$. It is obvious to get that, the eigenvectors corresponding to the eigenvalue 
$0$ are linearly independent, and
\begin{align*}
	\nabla\tilde{\lambda}_{i}(U, \xi)\cdot r_i(U, \xi)=0, \quad i=1,2, \ \forall u\in \Omega_\varepsilon, \ \forall\xi\in \mathbb{S}^1,
\end{align*}
For $3\leqslant i\leqslant 6,$ the corresponding right eigenvectors can be chosen as
\begin{align*}
	r_i(U, \xi)&={\left(-
		\frac{c_{12}\xi_1}{\tilde{\lambda}_i(\tilde{\lambda}_i^2-c_{11})},
		-\frac{c_{12}\xi_2}{\tilde{\lambda}_i(\tilde{\lambda}_i^2-c_{11})},
		-\frac{\xi_1}{\tilde{\lambda}_i},
		-\frac{\xi_2}{\tilde{\lambda}_i},
		\frac{c_{12}}{\tilde{\lambda}_i^2-c_{11}},
		1
		\right)^{\large {\top}}},
\end{align*}
and then
\begin{align*}
	\nabla\tilde{\lambda}_{3,4}&=\frac{1}{2\tilde{\lambda}_{3,4}}\left(-\frac{1}{2}(\nabla c_{11}+\nabla c_{22})-\frac{1}{4\sqrt{(c_{11}-c_{22})^2+4c_{12}c_{21}}}(2(c_{11}-c_{22})(\nabla c_{11}-\nabla c_{22})+4c_{21}\nabla c_{12}+4c_{12}\nabla c_{21})\right)\\
	&=\frac{1}{2\tilde{\lambda}_{3,4}}\left(\frac{1}{2}\left(\begin{matrix}
		-\xi_2^2\\
		\xi_1\xi_2\\
		\xi_1\xi_2\\
		-\xi_1^2\\
		0\\
		0
	\end{matrix}\right)-\frac{1}{4\sqrt{(c_{11}-c_{22})^2+4c_{12}c_{21}}}\left(\begin{matrix}
		2(c_{11}-c_{22})\xi_2^2-4c_{12}\xi_1\xi_2\\
		-2(c_{11}-c_{22})\xi_1\xi_2+4c_{12}\xi_1^2\\
		2(c_{11}-c_{22})\xi_1\xi_2+4c_{21}\xi_2^2\\
		-2(c_{11}-c_{22})\xi_1^2-4c_{21}\xi_1\xi_2\\
		0\\
		0
	\end{matrix}\right)\right),
\end{align*}
\begin{align*}
	\nabla\tilde{\lambda}_{5,6}&=\frac{1}{2\tilde{\lambda}_{5,6}}\left(-\frac{1}{2}(\nabla c_{11}+\nabla c_{22})+\frac{1}{4\sqrt{(c_{11}-c_{22})^2+4c_{12}c_{21}}}(2(c_{11}-c_{22})(\nabla c_{11}-\nabla c_{22})+4c_{21}\nabla c_{12}+4c_{12}\nabla c_{21})\right)\\
	&=\frac{1}{2\tilde{\lambda}_{5,6}}\left(\frac{1}{2}\left(\begin{matrix}
		-\xi_2^2\\
		\xi_1\xi_2\\
		\xi_1\xi_2\\
		-\xi_1^2\\
		0\\
		0
	\end{matrix}\right)+\frac{1}{4\sqrt{(c_{11}-c_{22})^2+4c_{12}c_{21}}}\left(\begin{matrix}
		2(c_{11}-c_{22})\xi_2^2-4c_{12}\xi_1\xi_2\\
		-2(c_{11}-c_{22})\xi_1\xi_2+4c_{12}\xi_1^2\\
		2(c_{11}-c_{22})\xi_1\xi_2+4c_{21}\xi_2^2\\
		-2(c_{11}-c_{22})\xi_1^2-4c_{21}\xi_1\xi_2\\
		0\\
		0
	\end{matrix}\right)\right),
\end{align*}
where
\begin{align*}
	\nabla c_{11}&={\left(
		\xi_2^2,-\xi_1\xi_2,0,0,0,0
		\right)^{\large {\top}}}, \quad\nabla c_{22}={\left(
		0,0,-\xi_1\xi_2, \xi_1^2,0,0
		\right)^{\large {\top}}}, \\
	\nabla c_{21}&={\left(
		-\xi_1\xi_2, \xi_1^2,0,0,0,0
		\right)^{\large {\top}}}, \quad\nabla c_{12}={\left(
		0,0, \xi_2^2,-\xi_1\xi_2,0,0
		\right)^{\large {\top}}}.
\end{align*}
Then, we could obtain
\begin{align*}
	\nabla\tilde{\lambda}_{i}(U, \xi)\cdot r_i(U, \xi)=0, \quad 3\leqslant i\leqslant6, \ \forall u\in \Omega_\varepsilon, \ \forall\xi\in \mathbb{S}^1.
\end{align*}
To sum up, we have
\begin{align*}
	\nabla\tilde{\lambda}_{i}(U, \xi)\cdot r_i(U, \xi)=0, \quad 1\leqslant i\leqslant6, \ \forall u\in \Omega_\varepsilon, \ \forall\xi\in \mathbb{S}^1.
\end{align*}
We conclude that all characteristic families of the system \eqref{system 2D} are linearly degenerate in the sense of Majda \cite{Majda84}.\par 
Next, we consider about the case for three dimensions. We use similar symbols without causing confusion. Similarly, we have
\begin{align}
\frac{\partial U}{\partial t}+\sum_{j=1}^3A_j(U)\frac{\partial U}{\partial x_j}=0, \label{system 3D}
\end{align}
where $U=(\partial_1u_1, \partial_2u_1, \partial_3u_1, \partial_1u_2, \partial_2u_2, \partial_3u_2, \partial_1u_3, \partial_2u_3, \partial_3u_3, \partial_tu_1, \partial_tu_2, \partial_tu_3)^{\top}$ and
\begin{align*}
A_j(U)&=\left(\begin{matrix}
	0&E_j\\
	B_j&0\\
\end{matrix}\right), \quad 1\leqslant i\leqslant 3,
\end{align*}
and
\begin{align*}
E_1=\left(\begin{matrix}
	-1&0&0\\
	0&0&0\\
	0&0&0\\
	0&-1&0\\
	0&0&0\\
	0&0&0\\
	0&0&-1\\
	0&0&0\\
	0&0&0
\end{matrix}\right), \quad
E_2=\left(\begin{matrix}
	0&0&0\\
	-1&0&0\\
	0&0&0\\
	0&0&0\\
	0&-1&0\\
	0&0&0\\
	0&0&0\\
	0&0&-1\\
	0&0&0
\end{matrix}\right), \quad
E_3=\left(\begin{matrix}
	0&0&0\\
	0&0&0\\
	-1&0&0\\
	0&0&0\\
	0&0&0\\
	0&-1&0\\
	0&0&0\\
	0&0&0\\
	0&0&-1
\end{matrix}\right),
\end{align*}
\begin{align*}
B_1&=\left(\begin{matrix}
	-\lambda-2\mu&2\partial_2u_1&2\partial_3u_1&0&-\lambda-\mu+2\partial_2u_2&2\partial_3u_2&0&2\partial_2u_3&-\lambda-\mu+2\partial_3u_3\\
	-\partial_2 u_1&-\partial_1u_1&0&	-\partial_2 u_2&-\mu-\partial_1u_2&0&-\partial_2 u_3&-\partial_1u_2&0\\
	-\partial_3 u_1&0&-\partial_1u_1&-\partial_3 u_2&-\mu&-\partial_1u_2&-\partial_3 u_3&0&-\partial_1u_3\\
\end{matrix}\right), \\
B_2&=\left(\begin{matrix}
	-\partial_2 u_1&-\mu-\partial_1 u_1&0&-\partial_2 u_2&-\partial_1 u_2&0&-\partial_2 u_3&-\partial_1 u_3&0\\
	-\lambda-\mu+2\partial_1u_1&0&2\partial_3u_1&2\partial_1u_2&-\lambda-2\mu&2\partial_3u_2&2\partial_1u_3&0&-\lambda-\mu+2\partial_3u_3\\
	0&-\partial_3 u_1&-\partial_2 u_1&0&-\partial_3 u_2&-\partial_2 u_1&0&-\mu-\partial_3 u_3&-\partial_2 u_1\\
\end{matrix}\right), \\
B_3&=\left(\begin{matrix}
	-\partial_3 u_1&0&-\mu-\partial_1 u_1&-\partial_3 u_2&0&-\partial_1 u_2&-\partial_3 u_3&0&-\partial_1 u_3\\
	0&-\partial_3u_1&	-\partial_2 u_1&0&-\partial_3u_2&	-\mu-\partial_2 u_2&0&-\partial_3u_3&-\partial_2 u_3\\
	-\lambda-\mu+2\partial_1u_1&2\partial_2u_1&0&2\partial_1u_2&-\lambda-\mu+2\partial_2u_2&0&2\partial_1u_3&2\partial_2u_3&-\lambda-2\mu
\end{matrix}\right).
\end{align*}
For $\xi=(\xi_1,\xi_2,\xi_3)\in \mathbb{S}^{2}$, denote
\begin{align*}
A(U, \xi)=\sum_{j=1}^3A_{j}(U)\xi_j
=\left(\begin{matrix}
	0&\sum_{j=1}^3\xi_jE_j\\
	\sum_{j=1}^3\xi_jB_j&0
\end{matrix}\right).
\end{align*}
Next, we calculate the corresponding eigenvalues $\{\tilde{\lambda}_i\}_{i=1}^{12}$ and eigenvectors $\{r_i\}_{i=1}^{12}$. It is easy to get
\begin{align*}
\tilde{\lambda}_i\Id-A(U, \xi)
&:=\left(\begin{matrix}
	\tilde{\lambda}_i\Id_9&-\sum_{j=1}^3\xi_jE_j\\
	-\sum_{j=1}^3\xi_jB_j&\tilde{\lambda}_i\Id_3
\end{matrix}\right),
\end{align*}
and then
\begin{align*}
|\tilde{\lambda}_i\Id_{12}-A(U, \xi)|
&=\tilde{\lambda}_i^6\left|\tilde{\lambda}_i^6\Id_2-C\right|,
\end{align*}
where 
\begin{align*}
C=\sum_{i,j=1}^3\xi_i\xi_jB_iE_j&=\left(c_{ij}\right)_{i,j=1}^3.
\end{align*}
Then we know 12 real eigenvalues of $A(U, \xi)$ satisfy that 
$
\tilde{\lambda}_{1,2,3,4,5,6}=0,
$
and $\tilde{\lambda}_{7,8}, \tilde{\lambda}_{9,10},$ and $\tilde{\lambda}_{11,12}$ are the the square root of the corresponding eigenvalue of $C$.
It is obvious to get that, for the eigenvalue $0,$ there are six linearly independent eigenvectors, and
\begin{align*}
\nabla \tilde{\lambda}_{i}\cdot r_{i}=0, \quad 1\leqslant i\leqslant 6, \ \forall u\in \Omega_\varepsilon, \ \forall\xi\in \mathbb{S}^2.
\end{align*}
For $7\leqslant i\leqslant 12,$ by doing the similar calculation as before, we can get
\begin{align*}
\nabla\tilde{\lambda}_{i}(U, \xi)\cdot r_i(U, \xi)=0, \quad 7\leqslant i\leqslant12, \ \forall u\in \Omega_\varepsilon, \ \forall\xi\in \mathbb{S}^2.
\end{align*}
To sum up, we have
\begin{align*}
	\nabla\tilde{\lambda}_{i}(U, \xi)\cdot r_i(U, \xi)=0, \quad 1\leqslant i\leqslant12, \ \forall u\in \Omega_\varepsilon, \ \forall\xi\in \mathbb{S}^2.
\end{align*}
We conclude that all characteristic families of the system \eqref{system 3D} are linearly degenerate in the sense of Majda \cite{Majda84}.
\section{Inverse divergence operator}\label{Inverse divergence operator}
In this part, we introduce the inverse divergence operators which is originally defined in \cite{DS13}.
\begin{df}[Leray projection]
	Let $\upsilon\in C^\infty(\T^3;\R^3)$ be a smooth vector field. Let 
	\begin{equation}
		\cQ\upsilon:=\nabla\psi+\fint_{\T^3}\upsilon\label{Qv},
	\end{equation}
	where $\psi\in C^\infty(\T^3)$ is the solution of
	$\Delta\psi=\Div\ \upsilon$,
	with $\fint_{\T^3}\psi=0$. Furthermore, let $\mathcal{P}=\cI-\cQ$ be the Leray projection onto divergence-free fields with zero average.
\end{df}
\begin{df}[Inverse divergence]\label{def of R}
	Let $\upsilon\in C^\infty(\T^3;\R^3)$ be a smooth vector field. We can define $\cR\upsilon$ to be the matrix-valued periodic function
	\begin{equation}
		\cR\upsilon:=\frac{1}{4}(\nabla\mathcal{P}u+(\nabla\mathcal{P}u)^{\top})+\frac{3}{4}(\nabla u+(\nabla u)^{\top})-\frac{1}{2}(\Div u)\Id, \label{def of R operator}
	\end{equation}
	where $u\in C^\infty(\T^3;\R^3)$ is the solution of
	$\Delta u=\upsilon-\fint_{\T^3}\upsilon$,
	with $\fint_{\T^3}u=0$.\label{inverse operator 1}
\end{df}
\begin{pp}
	For any $\alpha\in(0,1)$ and any $N\in\N,$ there exists a constant $C(N,\alpha)$ with the following properties.
	For the operator $\cR$ defined above, we have for $\upsilon\in C^{N+\alpha}(\T^3;\R^3)$ and $A\in C^{N+\alpha}(\T^3;\R^{3\times 3})$ 
	\begin{equation}
		\begin{aligned}
			&\left\|\cR\upsilon\right\|_{N+1+\alpha}\leqslant C(N, \alpha)\left\|\upsilon\right\|_{N+\alpha}, \\
			&\left\|\cR(\Div A)\right\|_{N+\alpha}\leqslant C(N, \alpha)\left\|A\right\|_{N+\alpha}.
		\end{aligned}\label{eest on R}
	\end{equation}
\end{pp}
\begin{proof}
	By the standard Schauder estimates,  for any $\phi, \psi:\T^3\rightarrow\R$ with
	\begin{align*}
		\left\{\begin{array} { l } 
			{ \Delta \phi = f , } \\
			{ \fint_{\T^3}\phi=0, }
		\end{array} \quad \text { and } \quad \left\{\begin{array}{l}
			\Delta \psi=\operatorname{div} F, \\
			\fint_{\T^3}\psi=0,
		\end{array}\right.\right.
	\end{align*}
	we have $
	\left\|\phi\right\|_{N+2+\alpha}\leqslant C(N, \alpha)\left\|f\right\|_{N+\alpha}, \left\|\psi\right\|_{N+1+\alpha}\leqslant C(N, \alpha)\left\|F\right\|_{N+\alpha},$ which yields \eqref{eest on R}.
\end{proof} 
\section{Some technical lemmas}
In this section, we introduce some lemmas given in \cite{BDSV19,DK22,GK22}. The proof for the following two lemmas can be found in \cite[Appendix]{BDSV19}.
\begin{lm}\cite[Proposition A.1]{BDSV19}
	Suppose $F:\Omega\rightarrow\R$ and $\Psi:\R^n\rightarrow\Omega$ are smooth functions for some $\Omega\subset\R^m$. Then, for each $N\in\Z_+,$ we have 
	\begin{equation}
		\begin{aligned}
			&\|\nabla^N(F\circ\Psi)\|_0\lesssim\|\nabla F\|_0\|\nabla\Psi\|_{N-1}+\|\nabla F\|_{N-1}\|\Psi\|_0^{N-1}\|\Psi\|_N, \\
			&\|\nabla^N(F\circ\Psi)\|_0\lesssim\|\nabla F\|_0\|\nabla\Psi\|_{N-1}+\|\nabla F\|_{N-1}\|\nabla\Psi\|_0^{N},
		\end{aligned}\label{est on Holider norm}
	\end{equation}
	where the implicit constants in the inequalities depends only on $n, m,$ and $N$.\label{$Holider norm of composition}
\end{lm}
\begin{lm}\cite[Proposition C.2]{BDSV19}.\label{est on int operator}
	Let $N\geqslant1$. Suppose that $a\in C^\infty(\T^3)$ and $\xi\in C^{\infty}(\T^3;\R^3)$ satisfies
	\begin{equation}
		\frac{1}{C}\leqslant|\nabla\xi|\leqslant C,
	\end{equation}
	for some constant $C>1$. Then, we have 
	\begin{equation}\label{est on Rae^ikxi}
		\Bigg|\int_{\T^3}a(x)e^{ik\cdot\xi}\rd x\Bigg|\lesssim_{C,N}\frac{\|a\|_N+\|a\|_0\|\nabla\xi\|_N}{|k|^N},
	\end{equation}
	and for the operator $\cR$ defined in Definition \ref{def of R operator}, we have
	\begin{equation}\label{est on R ae}
		\|\cR\left(a(x) e^{i k \cdot \xi}\right)\|_{\alpha} \lesssim_{C,N,\alpha} \frac{\|a\|_{0}}{|k|^{1-\alpha}}+\frac{\|a\|_{N+\alpha}+\|a\|_{0}\|\xi\|_{N+\alpha}}{|k|^{N-\alpha}}.
	\end{equation}
\end{lm}
\section{Estimates for nonlinear equations}
In this section, we introduce a lemma which gives the solutions of a system of polynomial equations and will be used in the construction of the perturbation. To maintain the completeness of the article, we will use the Newton iteration method to provide the proof.
\begin{lm}
	Consider the following equations
	\begin{equation}\label{polynomial equations}
		\left\{
		\begin{aligned}
			&A_1a_{1}^2+E_1a_{1}a_{2}+C_1a_1+B_2a_2=D_1, \\		&A_2a_{2}^2+E_2a_{1}a_{2}+C_2a_2+B_1a_1=D_2,
		\end{aligned}\right.
	\end{equation}
where $A_i,B_i,C_i,D_i$ and $E_i$ are constants. If there exist positive constants $\tilde{C},$ $c,$ and $\varepsilon$ such that 
\begin{align*}
	\sum_{i=1}^2\left(|A_i|+|B_i|+|D_i|+|E_i|+|C_i-c|\right)&\leqslant\tilde{C}\varepsilon,
\end{align*}
we could find $\varepsilon_0(\tilde{C},c)>0$ such that, for any $\varepsilon<\varepsilon_0,$ there exists an solution $(a_1,a_2)$ of equation \ref{polynomial equations} such that 
\begin{align*}
	|a_1|+|a_2|\leqslant C\varepsilon,
\end{align*}
 for some constant $C(\tilde{C},c)>0$.
\end{lm}
\begin{proof}
If $\varepsilon<\min(\frac{c}{1000},\frac{c}{1000\tilde{C}}, \frac{c}{1000\tilde{C}^2}, \frac{c^2}{512\tilde{C}}, \frac{1}{2}),$ we have 
\begin{align*}
	\sum_{i=1}^2\left(|A_i|+|B_i|+|D_i|+|E_i|\right)&\leqslant \frac{c}{100}, \quad \frac{c}{2}\leqslant|C_1|\leqslant 2c,\quad \frac{c}{2}\leqslant|C_2|\leqslant 2c.
\end{align*}
Let $a_{1,1}=D_1/C_1$ and $a_{2,1}=D_2/C_2,$ we have 
	\begin{equation}
	\left\{
	\begin{aligned}
	&A_1a_{1,1}^2+E_1a_{1,1}a_{2,1}+C_1a_{1,1}+B_2a_{2,1}=D_1+\epsilon_{1,1}, \\		&A_2a_{2,1}^2+E_2a_{1,1}a_{2,1}+C_2a_{2,1}+B_1a_{1,1}=D_2+\epsilon_{2,1},
	\end{aligned}\right.
\end{equation}
where $\epsilon_{1,1}=(A_1D_1^2)/C_1^2+(E_1D_1D_2)/(C_1C_2)+(B_2D_2)/C_2$ and $\epsilon_{1,2}=(A_2D_2^2)/C_2^2+(E_2D_1D_2)/(C_1C_2)+(B_1D_1)/C_1$. So we have the following estimates
\begin{align*}
	|a_{1,1}|+|a_{2,1}|\leqslant\frac{4\tilde{C}\varepsilon}{c}\leqslant\frac{1}{100}, \quad|\epsilon_{1,1}|+|\epsilon_{2,1}|\leqslant \frac{8(\tilde{C}\varepsilon)^2}{c}\leqslant\varepsilon.
\end{align*}
Assume $a_{1,n}$ and $a_{2,n}$ satisfy
	\begin{equation} \label{eq of an}
	\left\{
	\begin{aligned}
		&A_1a_{1,n}^2+E_1a_{1,n}a_{2,n}+C_1a_{1,n}+B_2a_{2,n}=D_1+\epsilon_{1,n}, \\		&A_2a_{2,n}^2+E_2a_{1,n}a_{2,n}+C_2a_{2,n}+B_1a_{1,n}=D_2+\epsilon_{2,n},
	\end{aligned}\right.
\end{equation}
and
\begin{align}
	|a_{1,n}|+|a_{2,n}|\leqslant 1.\label{est on an}
\end{align}
Then we could add a correction $\tilde{a}_{1,n}$ and $\tilde{a}_{2,n}$  which is the solution of 
\begin{equation}
	\left\{
	\begin{aligned}
		&(C_1+2A_1a_{1,n}+E_1a_{2,n})\tilde{a}_{1,n}+(B_2+E_1a_{1,n})\tilde{a}_{2,n}=-\epsilon_{1,n}, \\		&(B_1+E_2a_{2,n})\tilde{a}_{1,n}+(C_2+2A_2a_{2,n}+E_2a_{1,n})\tilde{a}_{2,n}=-\epsilon_{2,n}.
	\end{aligned}\right.
\end{equation}
Then we have
\begin{align*}
	\tilde{a}_{1,n}&=-\frac{\epsilon_{1,n}(C_2+2A_2a_{2,n}+E_2a_{1,n})-\epsilon_{2,n}(B_2+E_1a_{1,n})}{(C_1+2A_1a_{1,n}+E_1a_{2,n})(C_2+2A_2a_{2,n}+E_2a_{1,n})-(B_1+E_2a_{2,n})(B_2+E_1a_{1,n})}, \\
	\tilde{a}_{2,n}&=-\frac{\epsilon_{2,n}(C_1+2A_1a_{1,n}+E_1a_{2,n})-\epsilon_{1,n}(B_1+E_2a_{2,n})}{(C_2+2A_2a_{2,n}+E_2a_{1,n})(C_1+2A_1a_{1,n}+E_1a_{2,n})-(B_1+E_2a_{2,n})(B_2+E_1a_{1,n})},
\end{align*}
and
\begin{align*}
	|\tilde{a}_{1,n}|+|\tilde{a}_{2,n}|\leqslant \frac{16c}{c^2}\left(|\epsilon_{1,n}|+|\epsilon_{2,n}|\right)\leqslant \frac{16}{c}\left(|\epsilon_{1,n}|+|\epsilon_{2,n}|\right).
\end{align*}
So the new items $a_{1,n+1}=a_{1,n}+\tilde{a}_{1,n}$ and $a_{2,n+1}=a_{2,n}+\tilde{a}_{2,n}$ satisfy
\begin{equation}
	\left\{
	\begin{aligned}
		&A_1a_{1,n+1}^2+E_1a_{1,n+1}a_{2,n+1}+C_1a_{1,n+1}+B_2a_{2,n+1}=D_1+\epsilon_{1,n+1}, \\		&A_2a_{2,n+1}^2+E_2a_{1,n+1}a_{2,n+1}+C_2a_{2,n+1}+B_1a_{1,n+1}=D_2+\epsilon_{2,n+1},
	\end{aligned}\right.
\end{equation}
where 
\begin{equation}
	\begin{aligned}
		|\epsilon_{1,n+1}|=|A_1\tilde{a}_{1,n}^2+E_1\tilde{a}_{1,n}\tilde{a}_{2,n}|
		&\leqslant\frac{256\tilde{C\varepsilon}}{c^2}\left(|\epsilon_{1,n}|+|\epsilon_{2,n}|\right)^2<\frac{\left(|\epsilon_{1,n}|+|\epsilon_{2,n}|\right)^2}{2}, \\		|\epsilon_{2,n+1}|=|A_2\tilde{a}_{2,n}^2+E_2\tilde{a}_{1,n}\tilde{a}_{2,n}|
		&\leqslant\frac{256\tilde{C\varepsilon}}{c^2}\left(|\epsilon_{1,n}|+|\epsilon_{2,n}|\right)^2<\frac{\left(|\epsilon_{1,n}|+|\epsilon_{2,n}|\right)^2}{2}.
	\end{aligned}
\end{equation}
By repeating this process, we can get a sequence $(a_{1,n}=a_{1,1}+\sum_{i=1}^{n-1}\tilde{a}_{1,i},a_{2,n}=a_{2,1}+\sum_{i=1}^{n-1}\tilde{a}_{2,i}, \epsilon_{1,n}, \epsilon_{2,n})$ which satisfies the following estimates
\begin{align*}
&|\epsilon_{1,n}|+|\epsilon_{2,n}|< \varepsilon^{2^{n-1}}, \quad
|a_{1,n}|+|a_{2,n}|\leqslant \frac{4\tilde{C}\varepsilon}{c}+\frac{16}{c}\sum_{i=2}^{
\infty} \varepsilon^{2^{i-1}}\leqslant C(\tilde{C},c)\varepsilon\leqslant 1.
\end{align*}
Let $a_1=\lim_{n\rightarrow\infty}a_{1,n},a_2=\lim_{n\rightarrow\infty}a_{2,n},$ we complete the proof.
\end{proof}

\noindent{\bf Conflict of interest.} The authors declare that they have no conflict of interest.\ \\

\noindent{\bf Data availability statement.} Data sharing is not applicable to this article as no datasets were generated or analysed during the current study.

\end{document}